\documentclass[12pt,a4paper]{article}
\makeatletter
\newcommand*{\rom}[1]{\expandafter\@slowromancap\romannumeral #1@}
\makeatother
\usepackage{latexsym}
\usepackage{amssymb}
\usepackage{amsmath}
\usepackage{amsthm}
\usepackage{thmtools}
\usepackage{enumitem}
\usepackage{cite}
\usepackage[all]{xy}
\usepackage{framed}
\usepackage{hyperref}
\usepackage{graphicx}
\usepackage{parskip}
\usepackage{hyperref}
\usepackage{bm}
\usepackage[a4paper,margin=2cm]{geometry}
\usepackage{mathdots}
\usepackage{pifont}
\usepackage{marginnote}
\usepackage{xcolor}

\begingroup
    \makeatletter
    \@for\theoremstyle:=definition,remark,plain\do{%
        \expandafter\g@addto@macro\csname th@\theoremstyle\endcsname{%
            \addtolength\thm@preskip\parskip
            }%
        }
\endgroup

\makeatletter
\xydef@\txt@ii#1{\vbox{\vspace*{-5pt}%
 \let\\=\cr
 \tabskip=\z@skip \halign{\relax\hfil\txtline@@{##}\hfil\cr\leavevmode#1\crcr}}}
\makeatother

\theoremstyle{definition}
\newtheorem{thm}{Theorem}[section]
\newtheorem{lem}[thm]{Lemma}
\newtheorem{cor}[thm]{Corollary}
\newtheorem{defn}[thm]{Definition}
\newtheorem{propn}[thm]{Proposition}
\newtheorem*{thm*}{Theorem}

\newtheorem*{qn*}{Question}

\newtheorem{props}[thm]{Properties}

\newtheorem*{nts}{Note to self}

\theoremstyle{remark}
\newtheorem{rk}[thm]{Remark}

\newtheorem{ex}[thm]{Example}

\newtheoremstyle{custthm}{\parskip}{}{\normalfont}{}{\bfseries}{.}{ }{\thmname{#1} \thmnote{#3}}
\theoremstyle{custthm}

\newcommand{\gr}{\mathrm{gr}}

\renewcommand{\O}{\mathcal{O}}

\newcommand\blfootnote[1]{%
  \begingroup
  \renewcommand\thefootnote{}\footnote{#1}%
  \addtocounter{footnote}{-1}%
  \endgroup
}

{%
\begin{framed}\begin{nts}%
}%
{%
\end{nts}\end{framed}%
}

\begin{document}

\numberwithin{equation}{section}
\binoppenalty=\maxdimen
\relpenalty=\maxdimen

\title{New directions in the study of prime ideals in rational, nilpotent Iwasawa algebras}
\author{Adam Jones and William Woods}
\date{\today}
\maketitle
\begin{abstract}
\noindent Let $G$ be a nilpotent $p$-valuable (compact $p$-adic Lie) group. There is an ongoing investigation into the prime ideals of its completed group algebra (\emph{Iwasawa algebra}) $\mathbb{Z}_p[[G]]$, and there remains an open conjecture that they can all be proved to have a canonical standard form. We very this conjecture for several new classes of nilpotent groups, including those corresponding to the positive subalgebra of almost all classical and exceptional types, curiously excluding those of type $C$.
\end{abstract}

\tableofcontents

\section{Introduction}

Fix a prime $p$. and let $G$ be a $p$-valuable group in the sense of \cite{Laz65} (see Definition \ref{defn: p-valuable} below). This is a torsion-free, pro-$p$ group, isomorphic (as a topological group) to a closed subgroup of some $\mathrm{GL}_n(\mathbb{Z}_p)$ consisting of upper-triangular matrices congruent to the identity mod $p$ (if $p > 2$) or mod $4$ (if $p = 2$).

Let $R$ be a commutative pseudocompact ring. The \emph{completed group algebra} (also known to number theorists as the \emph{Iwasawa algebra}) of $G$ over $R$ is defined as
$$R[[G]] = \varprojlim_U R[G/U],$$
where $R[G/U]$ is the ordinary group algebra of the (finite, discrete) group $G/U$ with coefficients in $R$, as $U$ runs over the set of open normal subgroups of $G$.

Now fix a finite extension of fields $K/\mathbb{Q}_p$, with ring of integers $\O$, uniformiser $\pi$, and residue field $k = \O/\pi$. There is an ongoing investigation into the prime ideals of $\O G := \O[[G]]$, especially when $\O = \mathbb{Z}_p$, beginning with the paper \cite{ArdBro06}, and advanced in several subsequent articles, including \cite{reflexive}, \cite{ardakovInv}, \cite{jones-abelian-by-procyclic}, \cite{jones-primitive-ideals},\cite{Munster} and \cite{Mann}.

We say that a prime ideal $P$ of $\O G$ is \emph{standard} if we can find a closed normal subgroup $H$ of $G$ such that $G^*:=G/H$ is $p$-valuable and the image of $P$ under the surjection $\O G\to \O G^*$ is centrally generated. We say that $P$ is \emph{virtually standard} if $P\cap \O U$ is standard for some open subgroup $U$ of $G$ \cite[Definition 1.1]{jones-primitive-ideals}. 

The main conjecture in the study of prime ideals in Iwasawa algebras (see e.g. \cite[Conjecture 1.2]{jones-primitive-ideals}) roughly states that all prime ideals in $\O G$ are virtually standard. This is now known for large classes of prime ideals in various groups, but in general it remains open.

To simplify the conjecture, fix a prime ideal $P$ of $\O G$ and set $$P^\dagger = \ker(G\to \O G^\times\to (\O G/P)^\times),$$ a closed normal subgroup of $G$: then $P$ contains the augmentation ideal $\varepsilon = (P^\dagger - 1)\O G$, so we can pass to the quotient $$P^* = P/\varepsilon \subseteq \O G/\varepsilon \cong \O G^*,$$ where $G^* = G/P^\dagger$. Then $P^*$ is a prime ideal of $\O G^*$, and now we have $(P^*)^\dagger = 1$, so to prove $P$ is standard, we only need to prove that $P^*$ is centrally generated.

\textbf{Note:} The resulting group $G^*$ may no longer be $p$-valuable, as it may no longer be torsion-free. However, there is a finite normal subgroup $F^* \leq G^*$ such that the quotient $G^*/F^*$ \emph{is} $p$-valuable, and $P^* \cap \O F^*$ is the intersection of a $G$-orbit of prime ideals of $\O F^* \cong \O[F^*]$, the (ordinary) group algebra of $F^*$. Hence it is considered relatively harmless to assume that $F^* = 1$, or equivalently to assume without loss of generality from the start that $P^\dagger = 1$.

We say that a prime ideal $P$ of $\O G$ is \emph{faithful} if $P^\dagger=1$. Then setting $Z=Z(G)$ as the centre of $G$, then our main conjecture can be restated as follows.

\textbf{Conjecture A.} If $P$ is a faithful prime ideal of $\O G$, then $P = (P\cap \O Z)\O G$.

There are two natural cases to consider here.

\begin{enumerate}
\item If $\pi\in P$, we may quotient out by $(\pi)$ and pass to $kG := k[[G]] \cong \O G/(\pi)$. This version of Conjecture A was asked as \cite[Question O]{ArdBro06}, and has since been proven true for $G$ nilpotent \cite[Theorem 8.4]{ardakovInv}, and for $G$ of the form $\mathbb{Z}_p^d\rtimes\mathbb{Z}_p$ \cite[Theorem 1.4]{jones-abelian-by-procyclic}.
\item If $\pi\not\in P$, we may invert $\pi$ and pass to $KG :=K\otimes_\O \O G= \O G[\frac{1}{\pi}]$. Conjecture A is known for prime ideals of infinite codimension when $G$ is semisimple of type $A$ or $D$ \cite[Theorem 3]{Mann}, \cite[Theorem 2]{Mann2}, or for primitive ideals when $G$ is nilpotent \cite[Theorem A]{jones-primitive-ideals}.
\end{enumerate}

(Warning: $KG$ is \emph{not} the same as $K[[G]]$. For example, if we take $\mathbb{Z}_p[[x]]$ and invert $p$, we do \emph{not} get $\mathbb{Q}_p[[x]]$.)

In this paper, we restrict to case 2, about which much less is known. We will also assume  that $G$ is a nilpotent group, and we want to build on the results from \cite{jones-primitive-ideals} and try to prove Conjecture A for all prime ideals, rather than just primitive ones.

In \S \ref{sec: control theorem}, we will recall a certain characteristic abelian subgroup $A(G)$ of $G$, originally constructed in \cite[\S 3.5]{jones-primitive-ideals}, containing the centre $Z$. When approaching the conjecture for general prime results, the strongest result to date is:

\textbf{Theorem.} \cite[Theorem B]{jones-primitive-ideals} Setting $A:=A(G)$, if $P$ is a faithful prime ideal of $KG$, then $P = (P \cap KA)KG$.

Of course, if $A(G)=Z$, then Conjecture A follows immediately from this theorem. There are cases when this is true, for example when $G$ has nilpotency class 2, but it is sadly false in the vast majority of cases of interest. Indeed, whenever the second centre $Z_2(G)$ is abelian (which is the case for all nilpotent groups arising as positive root subgroups) then $A(G)$ is strictly larger than the centre.

Moreover, the construction of $A(G)$ relies on an iterative procedure (which can be of arbitrary finite length), and so it can be time-consuming to compute.

In this note, we will define a smaller subgroup $B(G)$ of $G$, with $Z\subseteq B(G)\subseteq A(G)$. This subgroup is much more explicitly constructed, and can often be computed more easily. Our main result is very similar:

\textbf{Theorem 1.} (Theorem \ref{thm: faithful primes are controlled by B}) Setting $B:=B(G)$, if $P$ is a faithful prime ideal of $KG$, then $P = (P \cap KB)KG$.\qed

Needless to say, how much of an improvement this result is depends entirely on how much smaller than $A(G)$ the subgroup $B(G)$ actually is. There are several cases when $B(G)=A(G)$ (including of course all the cases when $A(G)=Z$), so this result constitutes no improvement in this setting.

However, in \S \ref{sec: examples}, we will explore a number of examples of groups $G$ for which we can prove that $B(G)$ is a proper subgroup of $A(G)$, and in several of them we even have that $B(G)=Z$. Using our main theorem this is enough to prove Conjecture A for these groups, and in particular for the majority of those groups that arise from Lie theory:

\textbf{Theorem 2.} (Corollary \ref{cor: new cases}) If $G$ is nilpotent and Lie$(G)$ is the positive root subalgebra of a split-semisimple $\mathbb{Q}_p$-Lie algebra, containing no component of type $C$. Then for any faithful prime ideal $P$ of $KG$, $(P\cap KZ)KG=P$.

It is curious that this approach fails when $G$ is a positive root subgroup of type $C$, though we still believe the conjecture to be true in this case.

\textbf{Acknowledgements:} The first author is very grateful to the Heilbronn Institute for Mathematical Research for funding and supporting this research.

\section{Preliminaries}

In this section, we will recap some of the key concepts we will be using throughout the article. Firstly, with regard to notation, if $g$ and $h$ are elements of a group, we will write ${}^h g = hgh^{-1}$. To distinguish between group commutators and ring commutators, we will write the former as $(g,h) = {\vphantom{h}}^g h \cdot h^{-1} = g^{-1}h^{-1}gh$ and the latter as $[x,y] = xy - yx$ throughout this paper.

The following facts about group commutators are well known, and can be easily checked directly: we include them only for ease of reference later.

\begin{lem}\label{lem: commutators}
Let $G$ be a group, $n$ any positive integer, and $a,b,c\in G$ arbitrary elements.

\begin{enumerate}[label=(\roman*)]
\item $(ab, c) = {}^a (b,c) \cdot (a,c)$.
\item $(a^n, c) = {}^{a^{n-1}}(a,c) \cdot {}^{a^{n-2}}(a,c) \dots {}^{a}(a,c) \cdot (a,c)$.\qed
\end{enumerate}
\end{lem}

\subsection{\(p\)-valuable groups}

\begin{defn}\label{defn: p-valuation}
\cite[Ch. III, D\'efinition 2.1.2]{Laz65}
A \emph{$p$-valuation} $\omega$ on a group $G$ is a function $\omega: G\to \mathbb{R}\cup\{\infty\}$ such that, for all $g,h\in G$,
\begin{enumerate}[label=(\roman*)]
\item $\omega(gh^{-1}) \geq \min\{\omega(g), \omega(h)\}$,
\item $\omega((g,h)) \geq \omega(g) + \omega(h)$,
\item $\omega(g^p) = \omega(g) + 1$,
\item $\omega(g) > (p-1)^{-1}$, and
\item $\omega(g) = \infty \Leftrightarrow g = 1$.
\end{enumerate}

Moreover, we say that the $p$-valuation $\omega$ is \emph{saturated} if for all $g\in G$ such that $\omega(g)>\frac{1}{p-1}+1$, there exists $h\in G$ with $g=h^p$.
\end{defn}

\textbf{Note:} If $G$ carries a $p$-valuation then it is torsionfree, and if $g^p=h^p$ for any $g,h\in G$ then $g=h$. In particular, if $\omega$ is saturated and $\omega(g)>\frac{1}{p-1}+1$, then there is a \emph{unique} element $h$ such that $h^p=g$, we write $h=g^{p^{-1}}$.

Defined in this way, a $p$-valuation is a kind of group filtration. More precisely, we isolate the following properties from \cite{Laz65}.

\begin{props}
\label{props: p-valuations}
Let $G$ be a group and $\omega$ a fixed $p$-valuation on $G$.

\begin{enumerate}
\item Given any $\nu\in\mathbb{R}$, write $G_\nu = \omega^{-1}([\nu, \infty])$ and $G_{\nu^+} = \omega^{-1}((\nu, \infty])$. These are normal subgroups of $G$, and the quotient group $\gr_\nu(G) = G_\nu/G_{\nu^+}$ is abelian, so we will write its group operation as addition.
\item Define the \emph{associated graded group} $\displaystyle \gr(G) = \bigoplus_{\nu\in\mathbb{R}} \gr_\nu(G)$ \cite[Ch. II, 1.1.7.2]{Laz65}. It follows from Definition \ref{defn: p-valuation}(iii) that $\gr(G)$ is an $\mathbb{F}_p$-module, and any non-identity element $g\in G$ gives a nonzero homogeneous element $\gr(g) := g + G_{\mu^+}\in \gr_\mu (G)$, where $\mu = \omega(g)$. Moreover, if $\omega(g) = \mu$ and $\omega(h) = \nu$, we may define a bracket operation on $\gr(G)$:
$$[\gr(g), \gr(h)] := (g,h) + G_{(\mu+\nu)^+},$$
which extends bilinearly to an $\mathbb{F}_p$-Lie algebra structure on $\gr(G)$ \cite[Ch. II, 1.1.7]{Laz65}.
\item There is also an $\mathbb{F}_p$-linear endomorphism $\pi: \gr(G) \to \gr(G)$, uniquely defined by its action $\pi(\gr(g)) = \gr(g^p)$ on homogeneous elements. This makes $\gr(G)$ into an $\mathbb{F}_p[\pi]$-Lie algebra \cite[Ch. II, 2.1.1]{Laz65} which is torsion-free \cite[Ch. I, Th\'eor\`eme 1.2.3]{Laz65}, and we will say that the \emph{rank} of $G$ is the rank of $\gr(G)$ as a $\mathbb{F}_p[\pi]$-module \cite[D\'efinition 2.1.3]{Laz65}.
\item There is also a notion of \emph{completion} \cite[Ch. II, 1.1.5]{Laz65}: set $\displaystyle\widehat{G} = \varprojlim_{\nu\in\mathbb{R}} G/G_\nu$, and say that $G$ is \emph{complete} (with respect to $\omega$) if the natural map $G\to \widehat{G}$ is an isomorphism. Equivalently, we could define a metric on $G$ by $d(g,h):=\omega(g^{-1}h)$, and $G$ is complete if and only if it is topologically complete with respect to this metric.
\end{enumerate}
\end{props}

\textbf{Note:} If $G$ is complete with respect to $\omega$ then for any $p$-adic integer $\lambda$, if we write $\lambda:=\underset{n\in\mathbb{N}}{\sum}\lambda_np^n$ with $0\leq\lambda_n<p$, then we can define $g^\lambda$ for any $g\in G$ to be the limit of the sequence $\overset{N}{\underset{n=1}{\prod}}g^{\lambda_np^n}$ as $N\rightarrow\infty$.

\begin{defn}\cite[Definition 4.5]{ardakovInv}
Let $G$ be a group and $\omega$ a fixed $p$-valuation. The \emph{degree} of an automorphism $\varphi$ of $G$ (with respect to $\omega$) is
$$\deg(\varphi) = \deg_\omega(\varphi) = \inf_{g\in G} \left\{ \omega(\varphi(g)g^{-1}) - \omega(g) \right\}.$$
We will say that $\varphi$ is \emph{bounded} if $\deg(\varphi) > (p-1)^{-1}$.
\end{defn}

Note that, if $\varphi$ is conjugation by some fixed $h\in G$, then by Definition \ref{defn: p-valuation}(ii), (iv), $\varphi$ is bounded with respect to \emph{any} $p$-valuation $\omega$. So, following Lazard, we now make the following definition:

\begin{defn}\label{defn: p-valuable}
\cite[Ch. III, Definition 3.1.6]{Laz65}
A group $G$ is \emph{$p$-valuable} if there exists a $p$-valuation $\omega$ with respect to which $G$ is complete and has finite rank. Moreover, we say that $G$ is $p$-saturable if $\omega$ is saturated.
\end{defn}

\begin{ex}\label{uniform}
We say $G$ is a \emph{uniform pro-$p$ group} if it a finitely generated, torsionfree pro-$p$ group, and $(G,G)\subseteq G^p$ \cite[Definition 4.1 and Theorem 4.5]{ddms}. In this case, we can define a $p$-valuation $\omega$ on $G$ by $$\omega(g):=\sup\{n\in\mathbb{N}:g\in G^{p^n}\}$$ With respect to $\omega$, $G$ is a $p$-valuable, and in fact $p$-saturable group.
\end{ex}

\begin{propn}
\cite[\rom{3}.3.3]{Laz65}
For every $p$-valuable group $G$, there exists a canonical $p$-saturable group $Sat(G)$ such that $G$ can be embedded as an open normal subgroup of $Sat(G)$.
\end{propn}

\subsection{Lazard's equivalence of categories}

Let $G$ be a $p$-valuable group, with $p$-valuation $\omega$, and let $\Lambda_G:=\mathbb{Z}_p G$ be the $\mathbb{Z}_p$-Iwasawa algebra of $G$. Then $\Lambda_G$ carries a natural topology inherited from $\omega$. Specifically, we know using \cite[\rom{3}.2.3.6]{Laz65} that there exists a separated ring filtration $w$ on the group algebra $\mathbb{Z}_p[G]$ such that 
\begin{itemize}
\item $w$ restricts to the $p$-adic valuation $v_p$ on $\mathbb{Z}_p$,
\item $w(g-1)>1$ for all $g\in G$,
\item the associated graded ring $\gr_w\mathbb{Z}_p[G]$ is isomorphic to the enveloping algebra of the Lie algebra $\gr_\omega(G)$, and 
\item $\Lambda_G$ arises as the topological completion of $\mathbb{Z}_p[G]$ with respect to $w$.
\end{itemize}

Therefore, we can of course extend $w$ to $\Lambda_G$. Moreover, $w$ extends further to the rational Iwasawa algebra $ \mathbb{Q}_pG=\mathbb{Q}_p\otimes_{\mathbb{Z}_p} \Lambda_G$, where we define $w(p^{-m}\otimes r):=w(r)-m$. However, $\mathbb{Q}_pG$ is no longer complete with respect to $w$, so we can pass to the completion $\widehat{\mathbb{Q}_pG}$ of $\mathbb{Q}_pG$ with respect to $w$.

Now, inside $\Lambda_G$, $w(g-1)>1$, so $w((g-1)^n)\geq n\geq v_p(n)=w(n)$, so it follows that the sequence $\left(\frac{(-1)^{n+1}}{n}(g-1)^n\right)$ converges to 0 in $\mathbb{Q}_pG$, and hence the logarithm series $\log(g):=\underset{n\in\mathbb{N}}{\sum}\frac{(-1)^{n+1}}{n}(g-1)^n$ converges in $\widehat{\mathbb{Q}_pG}$.

Now, if $\omega$ is a saturated $p$-valuation, we define the \emph{$\mathbb{Z}_p$-Lie lattice of $G$} to be the set $$\log(G):=\{\log(g):g\in G\}\subseteq\widehat{\mathbb{Q}_pG}$$

\begin{thm}\cite[\rom{4}.3.2.3]{Laz65} $ $

For any $p$-saturable group $G$, $\log(G)$ is a free $\mathbb{Z}_p$-Lie subalgebra of $\widehat{\mathbb{Q}_pG}$ of finite rank, equal to the rank of $G$. Moreover, for any $g,h\in G$, $\lambda\in\mathbb{Z}_p$: 
\begin{itemize}
\item $\lambda\log(g)=\log(g^\lambda)$

\item $\log(g)+\log(h)=\log\left(\underset{n\rightarrow\infty}{\lim}{\left(g^{p^n}h^{p^n}\right)^{p^{-n}}}\right)$

\item $[\log(g),\log(h)]=\log\left(\underset{n\rightarrow\infty}{\lim}{\left(g^{p^n},h^{p^n}\right)^{p^{-2n}}}\right)$
\end{itemize}
\end{thm}

In fact, the functor $G\mapsto\log(G)$ defines an equivalence (and in fact isomorphism) of categories between the category of $p$-saturable groups and a certain category of free $\mathbb{Z}_p$-Lie algebras known as \emph{saturable Lie algebras} \cite[\rom{4}.3.2.6]{Laz65}. This is known as \emph{Lazard's equivalence of categories}.

\begin{ex}
If $G$ is a uniform pro-$p$ group, then if we let $\mathcal{L}:=\log(G)$, then $\mathcal{L}$ is a \emph{powerful $\mathbb{Z}_p$-Lie algebra}, i.e. $[\mathcal{L},\mathcal{L}]\subseteq p\mathcal{L}$.

Moreover, the logarithm functor restricts to an equivalence of categories between the category of uniform pro-$p$ groups and the category of powerful $\mathbb{Z}_p$-Lie algebras of finite rank.
\end{ex}

Lazard's equivalence will prove useful when looking for examples of $p$-valuable groups, since it is practically much easier to start with a powerful $\mathbb{Z}_p$-Lie algebra and pass to its image under the inverse functor $\exp$ rather than look for examples directly.

Finally, we define the \emph{Lie algebra} of a $p$-valuable group $G$ to be the $\mathbb{Q}_p$-Lie algebra $$\mathfrak{g}_G=Lie(G):=\mathbb{Q}_p\otimes_{\mathbb{Z}_p}\log(Sat(G))$$ Unfortunately, the functor $G\to \mathfrak{g}_G$ is no kind of equivalence of categories, and there are plenty of examples of non-isomorphic groups $G,H$ such that $\mathfrak{g}_G=\mathfrak{g}_H$ (for example, if $H$ is any open subgroup of $G$). But note that whenever $G$ is $p$-saturable, $\log(G)$ is a full rank $\mathbb{Z}_p$-Lie lattice in $\mathfrak{g}_G$.

\begin{rk}
If $G$ is not $p$-saturable, then $\log(G)\subseteq\widehat{\mathbb{Q}_pG}$ need not be closed under addition or the Lie bracket. Still, we can embed $\log(G)$ as a subset of $\log(Sat(G))$, so it will still span $\mathfrak{g}_G$ as a $\mathbb{Q}_p$-vector space.
\end{rk}

\begin{lem}\label{lem: isolated subgroups to ideals}
For any $p$-valuable group $G$, the map $H\mapsto \mathfrak{g}_H$ defines an inclusion preserving bijection between the set of closed, isolated normal subgroups of $G$ and the set of ideals of $\mathfrak{g}_G$.\qed
\end{lem}

\subsection{Right ideals and controlling subgroups}

Let $G$ be a $p$-valuable group and let $I$ be a right ideal of $\O G$ (or $KG$). We say that a subgroup $H$ of $G$ \emph{controls} $I$ if $I$ is generated by a subset of $\O H$ (or $KH$), i.e. $I=(I\cap \O H)\O G$ (or $I=(I\cap KH)KG$).

\begin{lem}\label{lem: control in KG and OG}
There exist mutually inverse bijections

\centerline{
\xymatrix@R=2px{
{\left\{\begin{array}{c}
\text{prime ideals of }\O G\text{ containing }p\\
\end{array}\right\}}
\ar@/^/@{->}[r]^-\Phi\ar@/_/@{<-}[r]_-\Psi&
{\left\{\begin{array}{c}
\text{prime ideals of }KG\\
\end{array}\right\}}
}
}

where $\Phi(Q):=K\otimes_{\O} Q$ and $\Psi(P):=P\cap\O G$. Moreover, $\Phi(Q)^\dagger=Q^\dagger$, and $P$ is controlled by a closed, normal subgroup $H$ of $G$ if and only if $\Phi(P)$ is controlled by $H$.
\end{lem}

\begin{proof}

It is straightforward to show that every two-sided ideal $I$ of $KG$ can be written as $K\otimes_\O J$, where $J=I\cap \O G$, and that $I$ is prime if and only if $J$ is prime.

On the other hand, if $Q$ is a prime ideal of $\O G$ not containing $p$, then let $P:=K\otimes_\O Q$. It is clear that $Q\subseteq P\cap \O G$, and for all $r\in P\cap \O G$, there exists $n\in\mathbb{N}$ such that $p^n r\in Q$. But $p$ is central in $\O G$, so since $p\notin Q$ and $Q$ is prime, it follows that $r\in Q$. Therefore $P\cap\O G=Q$, so it follows that $\Phi$ and $\Psi$ are well defined, mutually inverse bijections.

Moreover, for all group elements $g\in G$, and all prime ideals $P$ of $KG$, since $g-1\in\mathbb{Z}[G]\subseteq\O G$, it follows that $g-1\in P$ if and only if $g-1\in P\cap\O G=\Psi(P)$, so $P^\dagger=\Psi(P)^\dagger$. Since $\Phi$ and $\Psi$ are mutually inverse, it follows immediately that $\Phi(Q)^\dagger=Q^\dagger$ for all prime ideals $Q$ of $\O G$ not containing $p$.

Now, if $H$ controls a prime ideal $Q$ of $\O G$ with $p\notin Q$, then for any $r\in\Phi(Q)$, we know that $p^nr\in Q$ for some $n\in\mathbb{N}$, so we can write $p^nr=s_1r_1+\dots+s_kr_k$ for some $s_i\in Q\cap\O H$, $r_i\in\O G$. Thus $r=p^{-n}\otimes (s_1r_1+\dots+s_kr_k)=(p^{-n}\otimes s_1)r_1+\dots+(p^{-n}\otimes s_k)r_k$, and since $p^{-n}\otimes s_i\in K\otimes (Q\cap\O H)=P\cap KH$, it follows that $H$ controls $P$.

Finally, let $P$ be a prime ideal of $KG$, and it remains to prove that if $H$ controls $P$ then $H$ controls $Q:=\Phi(P)=P\cap\O G$.

Using \cite[Theorem A]{controller}, we know that $H$ controls $Q$ if and only if $Q$ is controlled by any open normal subgroup of $G$ containing $H$. Since $P$ is also controlled by any such open normal subgroup, we may assume that $H$ is open in $G$, so let $\{g_1,\dots,g_k\}$ be a complete set of coset representatives for $H$ in $G$.

If $r\in Q$ then we can write $r=\underset{1\leq i\leq k}{\sum}r_ig_i$ for some $r_i\in\O H$. But $r\in P=(P\cap KH)KG$, so we know that $r_i\in P\cap KH$ for each $i\leq k$, thus $r_i\in P\cap KH\cap \O G=P\cap \O H=Q\cap\O H$. So $r\in (Q\cap \O H)\O G$, and $Q$ is controlled by $H$ as required.\end{proof}

If $I$ is a right ideal of $K G$, we define the \emph{controller subgroup} $I^\chi$ of $I$ to be the intersection of all open subgroups $U$ of $G$ controlling $I$, which is of course a closed subgroup of $G$.

\begin{lem}\label{lem: controller subgroup}
If $P$ is a prime ideal of $KG$, then $P^\chi$ is a closed, normal subgroup of $G$, and a subgroup $H$ of $G$ controls $P$ if and only if $P^\chi\subseteq H$. In particular, $P^\chi$ controls $P$.
\end{lem}

\begin{proof}

Clearly $P^\chi$ is a closed subgroup of $G$, and if an open subgroup $U$ of $G$ controls $P$ then for all $g\in G$, $(P\cap K^gU)KG=g(P\cap KU)KGg^{-1}=gPg^{-1}=P$, so $^gU$ controls $P$. It follows that the intersection $P^\chi$ of all open subgroups controlling $P$ is a normal subgroup of $G$.

Now, using Lemma \ref{lem: control in KG and OG}, we can write $P=K\otimes_\O Q$, where $Q=P\cap\O G$ is a prime ideal of $\O G$, and a closed subgroup of $G$ controls $P$ if and only if it controls $Q$, and hence $P^\chi=Q^\chi$. But by \cite[Theorem A]{controller}, a closed subgroup of $G$ controls $Q$ if and only if it contains $Q^\chi=P^\chi$ as required.\end{proof}



\begin{defn}\label{defn: non-splitting right ideals} $ $

\begin{enumerate}[label=(\roman*)]
\item \cite[Definition 5.5]{ardakovInv} A prime ideal $P$ of $KG$ is \emph{non-splitting} if $P\cap KH$ is prime for any open normal subgroup $H\lhd G$ that controls $P$.
\item \cite[Definition 5.7]{ardakovInv} A right ideal $I$ of $KG$ is \emph{virtually non-splitting} if $I = PKG$ for some non-splitting prime ideal $P\lhd KH$ and some open subgroup $H \leq G$.
\end{enumerate}
\end{defn}

\begin{props}\label{props: non-splitting right ideals}
$ $

\begin{enumerate}[label=(\roman*)]
\item \cite[Proposition 4.3]{jones-control-theorem} If $P$ is non-splitting, then $P\cap KP^\chi$ is prime.
\item \cite[Theorem 3.17]{jones-primitive-ideals} Let $C$ be a closed subgroup of $G$. If all faithful virtually non-splitting right ideals of $KG$ are controlled by $C$, then all faithful prime ideals of $KG$ are controlled by $C$.
\end{enumerate}
\end{props}

Here is our first refinement:

\begin{propn}\label{propn: from Z2 to Z}
Let $G$ be a nontrivial $p$-valuable group, and suppose that for every normal subgroup $U$ of $G$, every faithful prime ideal of $KU$ is controlled by $Z_2(U)$. Then every faithful prime $P$ of $KG$ is controlled by $Z(G)$.
\end{propn}

\begin{proof}
First, suppose $P$ is non-splitting, so that $Q := P\cap KP^\chi$ is a faithful prime ideal of $KP^\chi$, and $Q^\chi = P^\chi$. Fix $g\in G \setminus Z$, and let $\varphi_g$ be the bounded automorphism of $G$ given by conjugation by $g$. Clearly $\varphi_g(Q) = Q$, and $\varphi$ centralises both $Z$ and $Q^\chi/Z$, but $Q$ is not controlled by any proper open subgroup of $Q^\chi$, so by Theorem \ref{thm: control theorem using bounded automorphisms} $\varphi_g$ must act trivially on $Q^\chi$. But since $g$ was arbitrary, this says that every $g\in G$ centralises $Q^\chi$, i.e. $Q^\chi \subseteq Z$.

Now suppose $I=PKG$ is a faithful virtually non-splitting right ideal, for $P$ a faithful prime ideal in $KU$, for $U$ an open normal subgroup of $G$. Then $P$ is controlled by $Z_2(U)$ by assumption, so using the argument above we see that $P$ is controlled by $Z(U)=Z(G)\cap U$, so $$(I\cap KZ(G))KG=(P\cap KZ(U))KU\cdot KG=PKG=I$$ It follows from Properties \ref{props: non-splitting right ideals} that all faithful prime ideals in $KG$ are controlled by $Z(G)$.\end{proof}



The basic control theorem we will use throughout this note is:

\begin{thm}\label{thm: control theorem using bounded automorphisms}
\cite[Theorem 3.16]{jones-primitive-ideals} Let $G$ be a $p$-valuable group and let $P$ be a faithful prime ideal of $KG$ such that $KG/P$ is infinite dimensional over $K$. Fix a nontrivial bounded automorphism $\varphi$ of $G$ and a closed central subgroup $\mathcal{Z}$ of $G$. 

If $\varphi(P) = P$, and $\varphi$ centralises both $\mathcal{Z}$ and $G/\mathcal{Z}$, then $P$ is controlled by a proper open subgroup of $G$.\qed
\end{thm}

The requirement that $P$ have infinite codimension in $KG$ may seem problematic, but we are principally interested in the case where $G$ is solvable, in which case this does not pose a problem:

\begin{propn}\label{propn: infinite codimension}
If $G$ is a non-abelian, solvable $p$-valuable group and $P$ is a faithful prime ideal of $KG$, then $KG/P$ is infinite dimensional over $K$
\end{propn}

\begin{proof}
Suppose for contradiction that $P$ is faithful and $KG/P$ is finite dimensional. Since it is also a prime ring, there must exist a finite dimensional, irreducible $KG$-module $M$ such that $P=Ann_{KG}M$.

The action of $KG$ on $M$ extends to the completion $\widehat{KG}$, which contains the Lie algebra $\log(G)$ of $G$ as a subalgebra. This Lie algebra is solvable, so by Lie's theorem there must exist a finite extension of fields $F/K$ such that $F\otimes_K M$ is completely solvable over $\log(G)$. In particular, $F\otimes_K M$ contains a submodule $N$ of codimension 1, so let $M':=(F\otimes_K M)/N$.

Now, the action of $KG$ on $M$ extends to an action of $FG$ on $F\otimes_K M$, and $N$ remains an $FG$-submodule. So $M'$ is a $FG$-module, and the natural map $M\to M'$ is a non-zero $KG$-module map. Moreover, since $M$ is irreducible as a $KG$-module, this map must be injective, so setting $Q:=Ann_{FG}M'$ we see that $Q\cap KG=Ann_{KG}M'\subseteq Ann_{KG} M=P$.

But $M'$ is one dimensional over $F$, so writing $M'=Fv$, there exists a Lie algebra homomorphism $\lambda:\log(G)\to F$ such that $x\cdot v=\lambda(x)v$ for all $x\in\log(G)$. Let $\mathfrak{h}:=\ker(\lambda)$; an ideal of $\log(G)$ acting by zero on $M'$, and since $\log(G)$ is non-abelian, $\mathfrak{h}\neq 0$. Let $H:=\exp(\mathfrak{h})$; a non-trivial normal subgroup of $G$, acting by the identity on $M'$, and hence $(H-1)KG\subseteq Q\cap KG\subseteq P$, contradicting faithfulness.

Thus we conclude that $P$ has infinite codimension in $KG$.\end{proof}


\section{Control theorems}\label{sec: control theorem}

Throughout this section, we will fix $G$ a $p$-valuable group.

\begin{defn} $ $
\begin{itemize}
    \item The \emph{lower central series} of $G$ is the chain of closed, isolated, characteristic subgroups: $1=Z_0(G)\subseteq Z_1(G)\subseteq Z_2(G)\subseteq\dots$, where for all $i>0$, $$Z_i(G):=\{g\in G:(g,G)\subseteq Z_{i-1}(G)\}$$

    \item If $H$ is a subgroup of $G$, define the \emph{centraliser} of $H$ to be the subgroup $$\mathbf{C}_G(H):=\{g\in G:(g,H)=1\}$$ Note that if $H$ is closed/isolated/characteristic then so is $C_G(H)$.
\end{itemize}
\end{defn}

\subsection{Old Control Theorem: The subgroup $A(G)$}

In \cite[section 3]{jones-primitive-ideals}, the first author applied Theorem \ref{thm: control theorem using bounded automorphisms} to prove a control theorem for faithful prime ideals in $KG$, which we will recap in full detail now.

\begin{defn}\label{defn: A-groups}
Define a descending sequence of subgroups $G=A_0 \supseteq A_1\supseteq A_2\supseteq\dots$ by $A_{i+1} := \mathbf{C}_{A_{i}}(Z_2(A_i))$. Set $A(G) := \bigcap_{i\in\mathbb{N}} A_i$.
\end{defn}

Note that each $A_i$ is a closed, isolated, characteristic subgroup of $G$, so it follows that the sequence $G=A_0\supseteq A_1\supseteq\dots$ must terminate, and thus $A(G)=A_n$ for some $n\in\mathbb{N}$.

\begin{lem}\label{lem: properties of A}
$A(G)$ is a proper subgroup of $G$ if and only if $Z(G)\neq Z_2(G)$, and if $G$ is nilpotent then $A(G)$ is abelian.
\end{lem}

\begin{proof}

Since $A_1=\mathbf{C}_{G}(Z_2(G))$, it is clear that this is a proper subgroup of $G$ if and only if $Z_2(G)$ is not central in $G$, i.e. if and only if $Z(G)\neq Z_2(G)$.

If $G$ is nilpotent, then for each $i\geq 0$, $A_i$ is nilpotent, so either $A_i$ is abelian or $Z(A_i)\neq Z_2(A_i)$. In the latter case, $A_{i+1}=\mathbf{C}_{A_i}(Z_2(A_i))$ is a proper subgroup of $A_i$. Therefore, either $A_i$ is abelian for some $i$, or the descending sequence $G=A_0\supseteq A_1\supseteq\dots$ does not terminate. But we know this sequence does terminate at $A(G)$, so this implies that $A_i$ is abelian for sufficiently high $i$, and hence $A(G)$ is abelian.\end{proof}

The following result is a direct application of Theorem \ref{thm: control theorem using bounded automorphisms}.

\begin{thm}\cite[Theorem B and proof in \S 3.5]{jones-primitive-ideals}\label{thm: faithful primes controlled by A(G)}
If $G$ is solvable and $P$ is a faithful prime ideal of $KG$, then $P$ is controlled by $A(G)$.\qed
\end{thm}

\textbf{Note:} The statement of \cite[Theorem B]{jones-primitive-ideals} required $G$ to be nilpotent, but since we know that all prime ideals in $KG$ have infinite codimension by Proposition \ref{propn: infinite codimension}, the proof goes through without issue.

\begin{cor}\label{cor: class 2}
If $Z_2(G)=G$ then all faithful prime ideals of $KG$ are controlled by $Z(G)$.
\end{cor}

\begin{proof}
By definition, $A_1:=\mathbf{C}_G(Z_2(G))=\mathbf{C}_G(G)=Z(G)$, so for all $i>1$, $$A_i=\mathbf{C}_{A_{i-1}}(Z_2(A_{i-1}))=\mathbf{C}_{Z(G)}(Z_2(Z(G)))=Z(G)$$ hence $A(G)=Z(G)$, so all faithful prime ideals of $KG$ are controlled by $Z(G)$ by Theorem \ref{thm: faithful primes controlled by A(G)}.\end{proof}

Unfortunately, the condition that $Z_2(G)=G$ is very restrictive, but there are few other examples known where $A(G)=Z(G)$. Indeed, the following result confirms that this is generally very false. 

\begin{lem}\label{Z_2 in A}
If $Z_2(G)$ is abelian and $Z_2(G)\neq Z(G)$, then $Z_2(G)\subseteq A(G)$ and $A(G)\neq Z(G)$.
\end{lem}

\begin{proof}

Since $Z_2(G)$ is abelian, $Z_2(G)\subseteq A_1=C_G(Z_2(G))$, and in fact $Z_2(G)\subseteq Z(A_1)$. For any $i\geq 1$, $Z(A_i)\subseteq C_G(Z_2(A_i))=A_{i+1}$, and since $Z(A_i)$ commutes with everything in $A_i\supseteq A_{i+1}$, it follows that $Z(A_i)\subseteq Z(A_{i+1})$.

So suppose $i\geq 1$ and $Z_2(G)\subseteq Z(A_i)$, then $Z_2(G)\subseteq Z(A_{i+1})\subseteq A_i$. So we see by induction that $Z_2(G)\subseteq A_n$ for all $n$, and hence $Z_2(G)\subseteq A(G)$.\end{proof}

Since $Z_2(G)$ is abelian for many nilpotent groups of interest, we would like a stronger control theorem.

\subsection{New Control Theorem: The subgroup \(B(G)\)}\label{subsec: B(G)}

Let $\mathcal{S}$ be the set of all subgroups $H$ of $G$ which satisfy the following property:
\begin{equation}\label{eqn: star}
\tag{$*$}
\text{For all } g\in G:\, \text{if } (g,(g,H)) = 1 \text{ then } (g,H) = 1.
\end{equation}
For example, any $H\leq Z(G)$ trivially satisfies \eqref{eqn: star}.

\begin{lem}\label{lem: there exists unique maximal B}
$\mathcal{S}$ contains a unique maximal element, i.e. among all subgroups $H$ of $G$ satisfying \eqref{eqn: star}, there is a unique maximal subgroup $B(G)$.
\end{lem}

\begin{proof}
Firstly, $1\in \mathcal{S}$, so by Zorn's lemma, $\mathcal{S}$ must contain a maximal element $B$. To show that $B$ is unique, we will show that if $H_1, H_2\in \mathcal{S}$, then $H_1H_2\in \mathcal{S}$. 

Take $g\in G$: if $(g,(g,H_1H_2)) = 1$, then $(g,(g,H_1)) = (g,(g,H_2)) = 1$ as $H_1,H_2\subseteq H_1H_2$. But since $H_1, H_2\in \mathcal{S}$, this implies that $(g,H_1) = (g,H_2) = 1$, and hence $(g, H_1H_2) = 1$ by Lemma \ref{lem: commutators}.\end{proof}

\begin{lem}\label{lem: properties of B}
$ $

\begin{enumerate}[label=(\roman*)]
\item $B(G)$ is a closed subgroup of $G$.
\item $B(G)$ is characteristic in $G$.
\item $B(G)$ is isolated in $G$. 
\item $B(G)$ is contained in $A(G)$.
\end{enumerate}
\end{lem}

\begin{proof}
$ $

\begin{enumerate}[label=(\roman*)]
\item Fix some $g\in G$, and assume that $(g,(g,\overline{B(G)})) = 1$. Then $(g,(g,B(G))) = 1$ as $B(G)\subseteq \overline{B(G)}$, and hence $(g,B(G)) = 1$ by definition. But since conjugation by $g$ is a continuous endomorphism of $G$, and elements of $\overline{B(G)}$ are limits of sequences in $B(G)$, we must also have $(g,\overline{B(G)}) = 1$. Hence $\overline{B(G)}$ also satisfies \eqref{eqn: star}, and so it must be equal to $B(G)$ by maximality.
\item For any automorphism $\varphi$ of $G$, the subgroup $\varphi(B(G))$ also satisfies \eqref{eqn: star}, and hence $\varphi(B(G))\subseteq B(G)$, so $B(G)\subseteq\varphi^{-1}(B(G))\subseteq B(G)$, forcing equality.
\item Let $H:=i_G(B(G))$ be the isolator of $B(G)$, i.e. $B(G)$ is open in $H$ and $H$ is a closed, isolated subgroup of $G$, and hence $H^{p^n}\subseteq B(G)$ for some $n\in\mathbb{N}$. 

Suppose $g\in G$ and $(g,(g,H))=1$. But we know that $B(G)\subseteq H$, so $(g,(g,B(G))=1$, and hence $(g,B(G))=1$ by \eqref{eqn: star}. This implies that $(g,h^{p^n})=1$ for all $h\in H$, i.e. $(ghg^{-1})^{p^n} = h^{p^n}$. But in a $p$-valuable group, this implies that $ghg^{-1} = h$ \cite[Ch. III, 2.1.4]{Laz65}, so that $(g, h) = 1$. Hence $(g,H)=1$ and $H$ also satisfies \eqref{eqn: star}. Thus we conclude that $B(G)\subseteq H\subseteq B(G)$, so $B(G)=H$ is isolated in $G$.
\item By definition, $B(G) \subseteq A_0 = G$. Suppose for induction that $B(G)\subseteq A_i$ for some $i\in\mathbb{N}$, and let $g\in Z_2(A_i)$ be arbitrary. Then $(g,(g,B(G))) \subseteq (g,(g,A_i)) = 1$, and so by definition $(g,B(G)) = 1$. In other words, $B(G)$ is a subgroup of $A_i$ that centralises every $g\in Z_2(A_i)$, and so $B(G) \subseteq \mathbf{C}_{A_i}(Z_2(A_i)) = A_{i+1}$.\qedhere
\end{enumerate}
\end{proof}

\begin{lem}\label{lem: B of open subgroup}
If $U$ is an open, normal subgroup of $G$ then $B(U)\subseteq B(G)$.
\end{lem}

\begin{proof}

We only need to show that $B(U)$ satisfies \eqref{eqn: star}, and it will follow from the definition of $B(G)$ that $B(U)\subseteq B(G)$. 

Firstly, since $U$ is open in $G$, there must exist $n\in\mathbb{N}$ such that $G^{p^n}\subseteq U$. So if $g\in G$ and $(g,(g,B(U)))=1$, then $g$ commutes with $(g,b)$ for every $b\in B(U)$, and hence $(g^m,b) = (g,b)^m$ for all $m$ by Lemma \ref{lem: commutators}. In particular, $g$ commutes with $(g^{p^n},b)$, and so $(g^{p^n},(g^{p^n},B(U))) = 1$.

But $g^{p^n}\in U$ so applying \eqref{eqn: star} for $B(U)$ we see that $(g^{p^n},B(U))=1$, and hence $(g,B(U))=1$ by \cite[Ch. III, 2.1.4]{Laz65}.\end{proof}

\begin{thm}\label{thm: faithful primes are controlled by B}
If $G$ is nilpotent and $P$ is a faithful prime ideal of $KG$, then $P$ is controlled by $B(G)$.
\end{thm}

\begin{proof}
Suppose first that $P$ is non-splitting, and let $H = P^\chi$. Set $Q = P\cap KH$, a faithful prime ideal of $KH$ by Properties \ref{props: non-splitting right ideals}.

We know that $H\subseteq A(G)$ by Theorem \ref{thm: faithful primes controlled by A(G)}, so since $G$ is nilpotent, $H$ is abelian by Lemma \ref{lem: properties of A}. Suppose for contradiction that $H\not\subseteq B(G)$: then there exists some $g\in G$ satisfying $(g,(g,H)) = 1$ but not $(g,H) = 1$. Fix such a $g\in G$, and let $\varphi: H\to H$ be the automorphism given by conjugation by $g$. Then there must exist some $h\in H$ such that $\varphi(h) \neq h$. 

On the other hand, for all $h\in H$, we have that $\varphi(h)h^{-1} = (g,h)$ commutes with $g$, and so $\varphi$ fixes $\varphi(h)h^{-1}$. Now applying Theorem \ref{thm: control theorem using bounded automorphisms} to the subgroup of $\varphi$-invariants in $H = Z(H)$ shows that $Q$ is controlled by a proper subgroup of $H$, and hence $P$ is controlled by a proper subgroup of $H=P^\chi$, contradicting Lemma \ref{lem: controller subgroup}.

Now suppose instead that $I$ is a faithful, virtually non-splitting right ideal of $KG$. Then let $U$ be an open subgroup of $G$ such that $I = PKG$ for some faithful non-splitting prime ideal $P$ of $KU$. Then $P$ is controlled by $B(U)$, i.e. $P=(P\cap KB(U))KU$. Thus 
\begin{align*}
I &= PKG\\
&= (P\cap KB(U))KUKG\\
&\subseteq (P\cap KB(U))KG\\
&\subseteq (I\cap KB(U))KG,
\end{align*}

But $B(U)\subseteq B(G)$ by Lemma \ref{lem: B of open subgroup}, so it follows that $I$ is controlled by $B(G)$. Applying Properties \ref{props: non-splitting right ideals}, we deduce that every faithful prime ideal of $KG$ is controlled by $B(G)$ as required.\end{proof}

\subsection{Lie theoretic interpretation}\label{subsec: Lie theory}

For now, let $\mathfrak{g}$ be any finite dimensional $\mathbb{Q}_p$-Lie algebra. Similar to the property \eqref{eqn: star} in \S \ref{subsec: B(G)}, we define the following property of ideals $\mathfrak{a}$ of $\mathfrak{g}$.
\begin{equation}\label{eqn: 2star}
\tag{$**$}
\text{for all } x\in \mathfrak{g}:\, \text{if } [x,[x,\mathfrak{a}]] = 0 \text{ then } [x,\mathfrak{a}] = 0.
\end{equation}
Also similar to \S \ref{subsec: B(G)}, we define $B(\mathfrak{g})$ to be the maximal ideal of $\mathfrak{g}$ satisfying \eqref{eqn: 2star}. It is very straightforward to prove that $B(\mathfrak{g})$ is uniquely defined, and that it contains $Z(\mathfrak{g})$.

Now, fix a nilpotent $p$-valuable group $G$, and let $\mathfrak{g}:=\mathfrak{g}_G$ be the Lie algebra of $G$.

\begin{propn}\label{propn: B(G) to B(g)}
For any closed, isolated normal subgroup $H$ of $G$, $H$ satisfies \eqref{eqn: star} if and only if $\mathfrak{h}:=\mathfrak{g}_H$ satisfies \eqref{eqn: 2star}. In particular, $B(\mathfrak{g})=\mathfrak{g}_{B(G)}$.
\end{propn}

\begin{proof}

We know by Lemma \ref{lem: isolated subgroups to ideals} that the map $H\mapsto\mathfrak{g}_H$ is an inclusion preserving bijection between closed, isolated normal subgroups of $G$ and ideals in $\mathfrak{g}$. Thus it suffices to prove the first statement, because since  $B(G)$ is closed, isolated and normal in $G$ by Lemma \ref{lem: properties of B}, and all subgroups satisfying \eqref{eqn: star} are contained in $B(G)$, it will follow that $\mathfrak{g}_{B(G)}$ satisfies \eqref{eqn: 2star}. Moreover, for any ideal $\mathfrak{h}$ satisfying \eqref{eqn: 2star}, $\mathfrak{h}=\mathfrak{g}_H$ for a normal subgroup $H$ satisfying \eqref{eqn: star}, and hence $H\subseteq B(G)$ so $\mathfrak{h}\subseteq \mathfrak{g}_{B(G)}$. In other words, $\mathfrak{g}_{B(G)}=B(\mathfrak{g})$.

So, we want to prove that for any closed, isolated normal subgroup $H$, $H$ satisfies \eqref{eqn: star} if and only if $\mathfrak{g}_H$ satisfies \eqref{eqn: 2star}. Let us assume first that $\mathfrak{g}_H$ satisfies \eqref{eqn: 2star}, and we are given $g\in G$ such that $(g,(g,H))=1$. Fixing any $h\in H$, this means that $g$ commutes with $(g,h^{p^n})$, and hence with $(g^{p^n},h^{p^n})=(g,h^{p^n})^{p^n}$ for all $n\in\mathbb{N}$ by Lemma \ref{lem: commutators}. But we know that $$[\log(g),\log(h)]=\log(\underset{n\rightarrow\infty}{\lim}(g^{p^n},h^{p^n})^{p^{-2n}})$$ and $g$ must commute with $(g^{p^n},h^{p^n})^{p^{-2n}}\in Sat(G)$ for all $n\in\mathbb{N}$. Hence $g$ commutes with $[\log(g),\log(h)]$, so $[\log(g),[\log(g),\log(h)]]=0$. Since this is true for all $h\in H$, it follows that $[\log(g),[\log(g),\log(H)]]=0$, and since $\mathfrak{h}$ is spanned by $\log(H)$, this means that $[\log(g),[\log(g),\mathfrak{h}]]=0$.

Applying \eqref{eqn: 2star}, we see that $[\log(g),\mathfrak{h}]=0$, so $\log(g)$ commutes with everything in $\log(H)\subseteq\mathfrak{h}$, and it follows immediately that $g=\exp(\log(g))$ commutes with everything in $H$, i.e. $(g,H)=1$ and thus $H$ satisfies \eqref{eqn: star}.

Conversely, if $H$ satisfies \eqref{eqn: star}, then $H\subseteq B(G)\subseteq A(G)$ is abelian by Lemma \ref{lem: properties of A}. So set $\mathfrak{h}:=\mathfrak{g}_H$, and $\mathfrak{h}$ is an abelian ideal of $\mathfrak{g}$. Suppose that $[x,[x,\mathfrak{h}]]=0$ for some $x\in\mathfrak{g}$, then choose $k\in\mathbb{N}$ such that $X:=p^kx\in\log(H)$, and let $g:=\exp(X)\in G$. Since $[X,[X,\log(H)]]=0$ we know that $g$ commutes with $[X,\log(H)]$.

Fixing $h\in H$, we want to prove that $g$ commutes with $(g,h)=ghg^{-1}h^{-1}$, and it suffices of course to prove that $g$ commutes with $hgh^{-1}$. But $$hgh^{-1}=h\exp(X)h^{-1}=\exp(hXh^{-1})$$ so we only need to show that $g$ commutes with $hXh^{-1}$. Writing $h=\exp(Y)=\underset{n\geq 0}{\sum}\frac{Y^n}{n!}$ and $h^{-1}=\exp(-Y)=\underset{m\geq 0}{\sum}\frac{(-1)^mY^m}{m!}$, we can write $hXh^{-1}=\underset{n,m\geq 0}{\sum}\frac{(-1)^{m}}{n!m!}Y^nXY^{m}$.

Now, since $\mathfrak{h}$ is abelian, $Y,[X,Y]\in\mathfrak{h}$ must commute. Thus an easy induction on $m$ shows that $Y^nXY^m=mY^{n+m-1}[X,Y]+Y^{n+m}X$, so 
\begin{align*}
hXh^{-1}&=\underset{n,m\geq 0}{\sum}\frac{(-1)^{m}}{n!m!}mY^{n+m-1}[X,Y]+\underset{n,m\geq 0}{\sum}\frac{(-1)^{m}}{n!m!}Y^{n+m}X\\&=\underset{n,m\geq 0}{\sum}\frac{(-1)^{m+1}}{n!m!}Y^{n+m}[X,Y]+\underset{n,m\geq 0}{\sum}\frac{(-1)^{m}}{n!m!}Y^{n+m}X\text{ (substituting }m-1\text{ with }m\text{ in the first sum)}\\&=\left(\underset{n,m\geq 0}{\sum}\frac{(-1)^{m}}{n!m!}Y^{n+m}\right)(X-[X,Y])
\end{align*}

But $\underset{n,m\geq 0}{\sum}\frac{(-1)^{m}}{n!m!}Y^{n+m}=hh^{-1}=1$, so $hXh^{-1}=X-[X,Y]$, so since $[X,[X,Y]]=0$, it follows that $X$ commutes with $hXh^{-1}$, and hence $g=\exp(X)$ commutes with $hXh^{-1}$ as required.

Thus we conclude that $(g,(g,h))=1$ for all $h\in H$, so applying \eqref{eqn: star} we deduce that $(g,H)=1$. So since $g$ commutes with everything in $H$, $X:=\log(g)$ commutes with everything in $\log(H)$, and hence $[x,\log(H)]=[p^{-n}X,\log(H)]=0$, so $[x,\mathfrak{h}]=0$ and $\mathfrak{h}$ satisfies \eqref{eqn: 2star}.\end{proof}

Moreover, since we know that $Z(\mathfrak{g})=\mathfrak{g}_{Z(G)}$, we deduce the following useful consequence of Theorem \ref{thm: faithful primes are controlled by B}.

\begin{cor}\label{cor: when B(G)=Z(G)}
If $G$ is a nilpotent $p$-valuable group and $B(\mathfrak{g}_G)=Z(\mathfrak{g}_G)$, then all faithful prime ideals of $KG$ are controlled by $Z(G)$.    
\end{cor}

\begin{proof}
Setting $\mathfrak{g}=\mathfrak{g}_G$, we know that $B(\mathfrak{g})=\mathfrak{g}_{B(G)}$ by Proposition \ref{propn: B(G) to B(g)}. But $B(G)$ is a closed, isolated normal subgroup of $G$ by Lemma \ref{lem: properties of B}, and so is $Z(G)$. So since $\mathfrak{g}_{Z(G)}=Z(\mathfrak{g})=B(\mathfrak{g})=\mathfrak{g}_{B(G)}$ and the assignment $H\mapsto\mathfrak{g}_H$ is injective for closed, isolated normal subgroups by Lemma \ref{lem: isolated subgroups to ideals}, it follows that $B(G)=Z(G)$. Using Theorem \ref{thm: faithful primes are controlled by B}, we deduce that all faithful prime ideals of $KG$ are controlled by $Z(G)$.\end{proof}

In \S \ref{sec: examples} below, we will use the Lie algebra to calculate $A(G)$ and $B(G)$ explicitly for some $p$-valuable nilpotent groups, and explore why this new control theorem is an improvement.

\section{Examples}\label{sec: examples}

Throughout this section, we will fix $\mathfrak{g}$ a finite dimensional, nilpotent, $\mathbb{Q}_p$-Lie algebra. We want to calculate $B(\mathfrak{g})$ for various examples.

\subsection{Positive root bases}

To first outline a possible method to calculating $B(\mathfrak{g})$, we want to focus on a particular class of examples.

\begin{defn}\label{defn: positive root basis}
Given a basis $\{x_i:i\in I\}$ for $\mathfrak{g}$, we say that this is a \emph{positive root basis} if for every pair $i,j\in I$, there exists $k:=k(i,j)\in I$ such that $[x_i,x_j]\in\mathbb{Q}_px_k$. 
\end{defn}

\textbf{Note:} If $[x_i,x_j]=0$, we may take $k(i,j)$ to be any element of $I$. For convenience, we will usually take $k(i,j)=i$ in this case. If $[x_i,x_j]\neq 0$ then $k(i,j)$ is uniquely defined.\\

\begin{ex}
The key examples of nilpotent Lie algebras carrying positive root bases come from Lie theory. Recall, if $\mathfrak{s}$ is a split semisimple Lie algebra over $\mathbb{Q}_p$, then the standard Cartan decomposition gives us $$\mathfrak{s}=\mathfrak{h}\oplus\underset{\alpha\in\Phi}\bigoplus\mathfrak{s}_{\alpha}$$ where $\mathfrak{h}$ is a Cartan subalgebra, $\Phi\subseteq\mathfrak{h}^\ast$ is the root system of $\mathfrak{s}$, and $\mathfrak{s}_{\alpha}$ is the one dimensional $\alpha$-eigenspace of the action of $\mathfrak{h}$ on $\mathfrak{s}$. Note that for all $\alpha,\beta\in\Phi$, $$[\mathfrak{s}_{\alpha},s_{\mathfrak{\beta}}]=\begin{cases} \mathfrak{s}_{\alpha+\beta} & \alpha+\beta\in\Phi\\ 0 & \text{otherwise}\end{cases}$$

Decompose $\Phi=\Phi^+\sqcup\Phi^-$, where $\Phi^+$ (resp. $\Phi^{-}$) consist of the positive (resp. negative) roots in $\Phi$ with respect to some ordering, and let $\mathfrak{n}^+:=\underset{\alpha\in\Phi^+}{\bigoplus}\mathfrak{s}_{\alpha}$ (resp. $\mathfrak{n}^-:=\underset{\alpha\in\Phi^-}{\bigoplus}\mathfrak{s}_{\alpha}$). Then $\mathfrak{n}^+,\mathfrak{n}^-$ are nilpotent subalgebras of $\mathfrak{s}$, and we may write $\mathfrak{s}=\mathfrak{n}^-\oplus\mathfrak{h}\oplus\mathfrak{n}^+$.


Let $\mathfrak{g}:=\mathfrak{n}^+$, and for each $\alpha\in\Phi^+$ let $x_{\alpha}$ be a generator of the one dimensional space $\mathfrak{s}_{\alpha}$. Then $\{x_\alpha:\alpha\in\Phi^+\}$ is a positive root basis for $\mathfrak{g}$ (hence the name), where the element $k(\alpha,\beta)$ from Definition \ref{defn: positive root basis} is defined by $$k(\alpha,\beta)=\begin{cases}
\alpha+\beta & \alpha+\beta\in\Phi\\ \alpha & \text{otherwise}
\end{cases}$$
\end{ex}

\

The following technical result provides a method of calculating $B(\mathfrak{g})$ whenever $\mathfrak{g}$ has a positive root basis.

\begin{propn}\label{propn: calculating B(g) using root bases}
Let $\{x_i:i\in I\}$ be a positive root basis for $\mathfrak{g}$, and suppose $B(\mathfrak{g})\subseteq\text{ Span}\{x_j:j\in J\}$ for some subset $J\subseteq I$. Suppose further that
\begin{itemize}
\item there exists $i\in I$ such that $[x_i,[x_i,x_j]]=0$ for all $j\in J$.

\item There exists a subset $\varnothing\neq S\subseteq J$, maximal with respect to the property that $[x_i,x_m]\neq 0$ for all $m\in S$.

\item $k(i,m)\neq k(i,m')$ if $m,m'\in S$ and $m\neq m'$.
\end{itemize}
Then there exists a proper subset $T$ of $J$ such that $B(\mathfrak{g})\subseteq\text{ Span}\{x_j:j\in T\}$.
\end{propn}

\begin{proof}
For each $m\in S$, write $[x_i,x_m]=\mu_{i,m}x_{k(i,m)}$, with $\mu_{i,m}\neq 0$. Since $[x_i,[x_i,x_j]]=0$ for all $j\in J$ and $B(\mathfrak{g})\subseteq\text{ Span}\{x_j:j\in J\}$, we know that $[x_i,[x_i,B(\mathfrak{g})]]=0$, so $[x_i,B(\mathfrak{g})]=0$ by \eqref{eqn: 2star}.

Moreover, by maximality of $S$, $[x_i,x_j]=0$ for all $j\in J$ with $j\notin S$. So if $x\in B(\mathfrak{g})$ and $x=\underset{j\in J}{\sum}\lambda_j x_j$, then $$0=[x_i,x]=\underset{j\in J}{\sum}\lambda_j [x_i,x_j]=\underset{m\in S}{\sum}\lambda_m [x_i,x_m]=\underset{m\in S}{\sum}\lambda_m\mu_{i,m}x_{k(i,m)}$$ so since $k(i,m)\neq k(i,m')$ if $m\neq m'$, we see using linear independence that $\lambda_m\mu_{i,m}=0$ for all $m\in S$ and hence $\lambda_m=0$. It follows that if $T:=J\backslash S$ then $x=\underset{j\in T}{\sum}\lambda_j x_j$. Since this is true for all $x\in B(\mathfrak{g})$, we deduce that $B(\mathfrak{g})\subseteq\text{ Span}\{x_j:j\in T\}$, and since $S$ is non-empty by assumption, $T$ is a proper subset of $J$.\end{proof}

Of course, the difficulty in utilising this result lies in finding the element $i\in I$ and the subset $S\subseteq J$, and this issue must be dealt with separately for each example.

\subsection{Positive root subalgebras}

Again, let $\mathfrak{s}$ be a split semisimple Lie algebra over $\mathbb{Q}_p$, with Cartan subalgebra $\mathfrak{h}$, root system $\Phi$ and decomposition $\mathfrak{s}=\mathfrak{n}^+\oplus\mathfrak{h}\oplus\mathfrak{n}^-$.

Let $\mathfrak{g}:=\mathfrak{n}^+$, and let $\{x_{\alpha}:\alpha\in\Phi^+\}$ be the associated positive root basis for $\mathfrak{g}$.


Now, recall that the positive system $\Phi^+$ contains a unique highest root $\theta$ (with respect to the Bruhat ordering), and the centre of $\mathfrak{g}=\mathfrak{n}^+$ is the one dimensional eigenspace $\mathfrak{s}_{\theta}$. The following technical results allow us to apply Proposition \ref{propn: calculating B(g) using root bases} to root bases of this form.

\begin{lem}\label{lem: eliminating root systems}
Suppose $\Phi$ is an irreducible root system, not of type $C$, and $\Psi\subseteq \Phi^+$ is any subset with $\theta\in\Psi$ but $\Psi\neq\{\theta\}$. Then there exists $\alpha\in\Phi^+$ such that $2\alpha+\beta\notin\Phi$ for all $\beta\in\Psi$, but $\alpha+\gamma\in\Phi^+$ for \emph{some} $\gamma\in\Psi$.
\end{lem}

\begin{proof}

As we will show, for many irreducible root systems $\Phi$, we can choose a positive system $\Phi^+$ such that for all $\alpha,\beta\in\Phi^+$, $2\alpha+\beta\notin\Phi$. More generally, it would suffice to prove that for all $\alpha\in\Phi^+$, $\beta\in\Psi$, $2\alpha+\beta\notin\Psi$, because in this case, since $\theta$ is the unique highest root, it follows from \cite[Lemma 3.2]{root} that for all $\gamma\in\Psi$ with $\gamma\neq\theta$, there exists $\alpha\in\Phi^+$ such that $\alpha+\gamma\in\Phi$. Since we are assuming $\Psi\neq\{\theta\}$ this completes the proof.

However, there are examples where there do exist pairs $(\alpha,\beta)$ of positive roots such that $2\alpha+\beta\in\Phi$, so we will need to treat these cases separately. We will examine and eliminate each irreducible root system in turn, and we begin each example by realising each positive system as a set of real vectors. In all the examples below, $\{e_i:i=1,\dots,n\}$ represents the standard basis for $\mathbb{R}^n$.

\begin{enumerate}

    \item Type $A_n$: We take $\Phi$ to be the set of all integer vectors in $\mathbb{R}^{n+1}$ of length $\sqrt{2}$ whose entries sum to 0. We can take $\Phi^+=\{e_i-e_j:1\leq i<j\leq n\}$.

    For every $\alpha,\beta\in\Phi^+$, writing $\alpha=e_i-e_{j}$, $\beta=e_k-e_\ell$, $2\alpha+\beta=2e_i-2e_j+e_k-e_\ell$, which can only lie in $\Phi^+$ if $i=\ell$ and $j=k$, implying that $\ell>k$, which is impossible. So $2\alpha+\beta\notin\Phi^+$ and the result follows as outlined above.\\

    \item Type $B_n$: Since we are assuming that $\Phi$ does not have type $C$, we know that it does not have type $B_2=C_2$, so we may assume that $n>2$. We can take $\Phi$ to be the set of integer vectors in $\mathbb{R}^n$ of length 1 or $\sqrt{2}$, and $$\Phi^+=\{e_i-e_j,e_k,e_\ell+e_m:i,j,n,k,\ell,m\leq n, i<j,\ell\neq m\}$$

    In this case, it is possible that $2\alpha+\beta\in\Phi^+$ for $\alpha,\beta\in\Phi^+$, e.g. $\alpha=e_2$, $\beta=e_1-e_2$, so we must look more carefully. However, if $\alpha=e_i-e_j$ for some $i<j$, then $2\alpha+\beta$ has either $e_i$ coefficient greater than 1, or $e_j$ coefficient less than $-1$, so $2\alpha+\beta\notin\Phi^+$. Similarly, if $\alpha=e_k+e_\ell$ for $k\neq \ell$, then $2\alpha+\beta$ has either $e_k$ or $e_\ell$ coefficient greater than 1, so $2\alpha+\beta\notin\Phi^+$. So if $2\alpha+\beta\in\Phi$ then $\alpha=e_k$ for some $k$.
    
    If $\Psi$ contains a root of the form $\gamma=e_i-e_j$ with $i<j$, then if $j<n$, take $\alpha=e_j+e_n$ and $\alpha+\gamma=e_i-e_n\in\Phi^+$, and if $j=n$ then take $\alpha=e_\ell+e_n$ for some $i\neq\ell<n$ (which we know exists since $n>2$), and $\alpha+\gamma=e_i+e_\ell\in\Phi^+$. In either case, since $2\alpha+\beta\notin\Phi$, for all $\beta\in\Psi$, the result follows. 
    
    Therefore, we may assume that $\Psi$ does not contain a root of the form $e_i-e_j$, so every root in $\Psi$ has the form $e_k$ or $e_\ell+e_m$. Therefore, if $\alpha=e_k$ then $2\alpha+\beta$ has $e_k$ coefficient at least 2 for every $\beta\in\Psi$, so $2\alpha+\beta\notin\Phi$. So $2\alpha+\beta\notin\Phi$ for all $\alpha\in\Phi^+$, $\beta\in\Psi$, which completes the proof by the argument outlined above.\\

    \item Type $D_n$: We take $\Phi$ to be the set of integer vectors in $\mathbb{R}^n$, $n\geq 4$, of length $\sqrt{2}$, and $$\Phi^+=\{e_i-e_j,e_\ell+e_m:i,j,\ell,m\leq n,i<j,\ell\neq m\}$$

    Again, for all $\alpha,\beta\in\Phi^+$, if $\alpha=e_i\pm e_j,\beta=e_k\pm e_\ell$, then $2\alpha+\beta=2e_i\pm 2e_j+e_k\pm e_\ell$. So this can only lie in $\Phi$ if $\alpha=e_i-e_j$ and $\beta=e_j-e_i$, which cannot both lie in $\Phi^+$. Therefore, $2\alpha+\beta\notin\Phi$ for all $\alpha,\beta\in\Phi^+$, which completes the proof.\\

    \item Type $E_8$: We take $\Phi$ to be the set of vectors in $\mathbb{R}^8$ of length $\sqrt{2}$, whose coordinates are either all integers or all half integers, and which sum to an even integer. Take
    \begin{align*}
        \Phi^+=&\{e_i-e_j,e_\ell+e_m,-(e_k+e_8):i<j\leq 8,\ell<m\leq 7,k\leq 7\}\\&\bigsqcup\left\{\frac{1}{2}\underset{1\leq i\leq 8}{\sum}(-1)^{n_i}e_i:n_8=1,\underset{1\leq i\leq 8}{\sum}n_i\in 4\mathbb{Z}\right\}
    \end{align*}

    Given $\alpha,\beta\in\Phi^+$, the $e_8$ coefficients of $\alpha$ and $\beta$ can be 0,$-\frac{1}{2}$ or $-1$. If the $e_8$ coefficient of $\alpha$ is $-1$ then the corresponding coefficient of $2\alpha+\beta$ is $-2,-2.5$ or $-3$, so $2\alpha+\beta\notin\Phi^+$. \\
    
    If the $e_8$ coefficient of $\alpha$ is $-\frac{1}{2}$, then $$\alpha=\frac{1}{2}\underset{1\leq i\leq 8}{\sum}(-1)^{n_i}e_i$$ so $2\alpha=\underset{1\leq i\leq 8}{\sum}(-1)^{n_i}e_i$ has $e_i$ coefficient $-1$. So if $2\alpha+\beta\in\Phi^+$ then the $e_8$ coefficient of $\beta$ is zero, so $\beta=\pm e_i\pm e_j$ for some $i\neq j$. Therefore, $2\alpha+\beta$ has at least 5 non-zero $e_k$-coefficients, and these coefficients are not $\pm\frac{1}{2}$, so $2\alpha+\beta\notin\Phi^+$ -- contradiction.\\

    Finally, if $\alpha$ has $e_8$ coefficient 0, then $2\alpha=\pm2e_k\pm2e_\ell$, so for any $\beta$ of the form $\frac{1}{2}\underset{1\leq i\leq 8}{\sum}(-1)^{n_i}e_i$, $2\alpha+\beta$ has $e_k$-entry either $-2.5,-1.5,1.5$ or $2.5$, so $2\alpha+\beta\notin\Phi$. Therefore, we may assume that $\alpha$ and $\beta$ have $e_8$ coefficient 0, and hence they lie in a positive system of type $D_7$, so the previous argument shows that $2\alpha+\beta\notin\Phi^+$ as required.

    \item Type $E_6$ and $E_7$: In both these cases, we can realise $\Phi$ as a subset of the root system $\Phi'$ of type $E_8$, and we can ensure $\Phi^+\subseteq\Phi'^+$, i.e. $\Phi^+$ is the set all the roots in $\Phi'^+$ that lie in $\Phi$. When considering $E_8$, we showed that for all $\alpha,\beta\in\Phi'^+$, $2\alpha+\beta\notin\Phi'$, so the same must hold for all $\alpha,\beta\in\Phi^+$.

    \item Type $F_4$: We can take $\Phi$ to be the set of all vectors in $\mathbb{R}^4$ of length 1 or $\sqrt{2}$ with either all entries integers or all entries half integers. Take
    \begin{align*}
    \Phi^+=&\{e_i-e_j,e_k,e_\ell+e_m,-(e_s+e_4),-e_4:i<j\leq 4,\ell<m\leq 3,s,k\leq 3\}\\&\bigsqcup\left\{\frac{1}{2}\underset{1\leq i\leq 4}{\sum}(-1)^{n_i}e_i:n_4=1\right\}
    \end{align*}

    Again, as with $B_n$, it is possible to find $\alpha,\beta\in\Phi^+$ such that $2\alpha+\beta\in\Phi$. However, by a similar argument to the case for $B_n$ and $E_8$, we can deduce that if $2\alpha+\beta\in\Phi$ then $\alpha\neq e_i-e_j,e_\ell+e_m$ or $-(e_s+e_4)$.

    
    Moreover, suppose there exists an element $\gamma$ of $\Psi$ of the form $\gamma=e_i-e_j,e_\ell+e_m,-(e_s+e_4)$ (assuming $\gamma\neq\theta= e_1-e_4$), then take $\alpha=e_1-e_i$ (or $e_j-e_4$ if $i=1$), $-(e_m+e_4)$ and $e_s+e_{s'}$ respectively (where $s\neq s'<4$), and we see that $2\alpha+\beta\notin\Phi$ for all $\beta\in\Psi$ but $\alpha+\gamma\in\Phi^+$ as required.

    Therefore, we can assume that for every $\beta\in\Psi$, $\beta=-e_4$, $e_1-e_4$, $e_k$ for some $k\leq 3$ or $\frac{1}{2}\underset{1\leq i\leq 4}{\sum}(-1)^{n_i}e_i$. From this description, it follows immediately that if $\alpha=e_k$ for some $k<4$ then $2\alpha+\beta\notin\Phi$ for all $\beta\in\Psi$.
    
    Also, if $\alpha=\frac{1}{2}\underset{1\leq i\leq 4}{\sum}(-1)^{m_i}e_i$ then $2\alpha+\beta$ has $e_4$-coefficient $-1.5,-2$ or $-2.5$ if $\beta\in\Psi$ and $\beta\neq e_k$ for some $k$, and if $\beta=e_k$ then $2\alpha+\beta$ is an integer vector with at least 3 non-zero entries. In either case, $2\alpha+\beta\notin\Phi$ for all $\beta\in\Psi$.

    Finally, if $\alpha=-e_4$ then $2\alpha+\beta$ has $e_4$ entry no more than $-2$ for all $\beta\in\Psi$, so $2\alpha+\beta\notin\Phi$, thus we conclude that for all $\alpha\in\Phi^+$, $\beta\in\Psi$, $2\alpha+\beta\notin\Phi^+$ as required.

    


    \item Type $G_2$: In this case, we can take $\Delta=\{\delta,\varepsilon\}\subseteq\mathbb{R}^2$ as a fundamental system, where $\delta=e_1$, and $\varepsilon$ makes an angle of $\frac{5\pi}{3}$ with $\delta$. Thus $\Phi^+$ consists of all vectors in $\Phi$ lying between $\delta$ and $\varepsilon$ in the plane, so we can write $$\Phi^+=\{\delta,\varepsilon,\varepsilon+\delta,3\delta+\varepsilon,2\delta+\varepsilon,3\delta+2\varepsilon\}$$ So if $\alpha,\beta\in\Phi^+$ and $2\alpha+\beta\in\Phi$, we must have that $(\alpha,\beta)\in\Phi^2$ is equal to $(\delta,\varepsilon),(\delta,\varepsilon+\delta)$ or $(\varepsilon+\delta,\delta)$. So we only need to eliminate the cases when $\Psi$ contains $\delta,\varepsilon$ or $\varepsilon+\delta$, in all other cases it follows that $2\alpha+\beta\notin\Phi$ for all $\alpha\in\Phi^+$, $\beta\in\Psi$.

    If $\varepsilon\in\Psi$, take $\alpha=3\delta+\varepsilon$, $\gamma=\varepsilon$. Then $2\alpha+\beta\notin\Phi$ for all $\beta\in\Psi$, but $\alpha+\gamma=3\delta+2\varepsilon\in\Phi$.
    
    If $\delta\in\Psi$ take $\alpha=\varepsilon$, $\gamma=\delta$. Then $2\alpha+\beta\notin\Phi$ for all $\beta\in\Psi$, but $\alpha+\gamma=\varepsilon+\delta\in\Phi^+$. 
    
    Finally, if $\varepsilon+\delta\in\Psi$, then take $\alpha=2\delta+\varepsilon$, $\gamma=\varepsilon+\delta$. Then $2\alpha+\beta\notin\Phi$ for all $\beta\in\Psi$, but $\alpha+\gamma=2\delta+\varepsilon+\varepsilon+\delta=3\delta+2\varepsilon\in\Phi^+$.\qedhere
    
\end{enumerate}

\end{proof}

\begin{propn}\label{propn: step 2 for irreducible root systems}
Suppose $\Phi$ is an irreducible root system, not of type $C$, and $\Psi\subseteq \Phi^+$ is any subset with $\theta\in\Psi$ but $\Psi\neq\{\theta\}$. Then there exists
\begin{itemize}
    \item $\alpha\in\Phi^+$ such that $[x_\alpha,[x_\alpha,x_\beta]]=0$ for all $\beta\in\Psi$.

    \item A non-empty subset $\Theta\subseteq\Psi$, maximal with respect to the property that $[x_\alpha,x_\beta]\neq 0$ for any $\beta\in\Theta$. 
    
    \item $k(\alpha,\beta)\neq k(\alpha,\gamma)$ for all $\beta,\beta'\in\Theta$ with $\beta\neq\beta'$.
\end{itemize}
\end{propn}

\begin{proof}

Using Lemma \ref{lem: eliminating root systems}, we know that we can find $\alpha\in\Phi^+$ such that $2\alpha+\beta\notin\Phi$ for all $\beta\in\Psi$, but $\alpha+\gamma\in\Phi$ for some $\gamma\in\Psi$. Therefore, $[x_\alpha,x_\gamma]\neq 0$, and $[x_\alpha,[x_\alpha,x_\beta]]\neq 0$ if and only if $\alpha+(\alpha+\beta)=2\alpha+\beta\in\Phi$. Therefore, we conclude that $[x_\alpha,[x_\alpha,x_\beta]]=0$ for all $\beta\in\Psi$.

The second and third statements follow immediately, because if $\Theta\subseteq\Psi$ is any subset maximal with respect to the property that $[x_\alpha,x_\beta]\neq 0$ (and we know this is non-empty because $\{\gamma\}$ defines such a subset), then $k(\alpha,\beta)=\alpha+\beta\in\Phi^+$ for all $\beta\in\Theta$, so if $\beta\neq\beta'$ then $k(\alpha,\beta)=\alpha+\beta\neq\alpha+\beta'=k(\alpha,\beta')$.\end{proof}

\begin{rk}
The proof of the results above assume we are making a particular choice of positive system in $\Phi$. However, if we assume further that $\mathfrak{s}$ is algebraic, then up to isomorphism the nilpotent subalgebras $\mathfrak{n}^+$ and $\mathfrak{n}^-$ do not depend on the choice of positive system $\Phi^+$, since all Borel subalgebras of $\mathfrak{s}$ are conjugate under the action of the adjoint algebraic group. Therefore, since allowing the positive system to vary does not change $\mathfrak{g}$ (up to isomorphism), the following result does not depend on a fixed choice of $\Phi^+$.
\end{rk}

\begin{thm}\label{thm: B(g) for g nilpotent subalgebra}
If $\mathfrak{g}$ is the positive nilpotent subalgebra of a split semisimple algebraic Lie algebra $\mathfrak{s}$ over $\mathbb{Q}_p$, and no irreducible component of $\mathfrak{s}$ has type $C$, then $B(\mathfrak{g})=Z(\mathfrak{g})$.
\end{thm}

\begin{proof}

Using \eqref{eqn: 2star}, it is easy to show that if $\mathfrak{g}=\mathfrak{g}_1\oplus\mathfrak{g}_2$ is a direct sum of two nilpotent Lie algebras, then $B(\mathfrak{g})=B(\mathfrak{g}_1)\oplus B(\mathfrak{g}_2)$ and $Z(\mathfrak{g})=Z(\mathfrak{g}_1)\oplus Z(\mathfrak{g}_2)$. So we may assume that $\mathfrak{g}$ is indecomposable, and hence the associated root system of $\mathfrak{s}$ is irreducible, and not of type $C$.

We know that $Z(\mathfrak{g})=\mathfrak{s}_\theta$, where $\theta$ is the highest root in $\Phi^+$, and we also know that $Z(\mathfrak{g})\subseteq B(\mathfrak{g})$. Choose a subset $\Psi\subseteq \Phi^+$ such that $B(\mathfrak{g})\subseteq\underset{\alpha\in\Psi}{\bigoplus}\mathfrak{s}_{\alpha}$, and we can assume that $\Psi$ is minimal with respect to this property.

Since $\mathfrak{s}_\theta\subseteq B(\mathfrak{g})$, we know that $\theta\in\Psi$, so let us assume for contradiction that $\Psi\neq\{\theta\}$. Then using Proposition \ref{propn: step 2 for irreducible root systems}, we know that 

\begin{itemize}
\item there exists $\alpha\in\Phi^+$ such that $[x_\alpha,[x_\alpha,x_\beta]]=0$ for all $\beta\in\Psi$,

\item we can choose a subset $\varnothing\neq\Theta\subseteq\Psi$, maximal with respect to the property that $[x_\alpha,x_\beta]\neq 0$ for all $\beta\in\Psi$, and 

\item $k(\alpha,\beta)\neq k(\alpha,\beta')$ for all $\beta,\beta'\in\Theta$, $\beta\neq\beta'$. 
\end{itemize}

Applying Proposition \ref{propn: calculating B(g) using root bases}, it follows that there exists a proper subset $\Psi'$ of $\Psi$ such that $B(\mathfrak{g})\subseteq\underset{\beta\in\Psi'}{\bigoplus}\mathfrak{s}_\beta$, contradicting minimality of $\Psi$. Therefore, $\Psi=\{\theta\}$, so $B(\mathfrak{g})=s_{\theta}=Z(\mathfrak{g})$ as requried.\end{proof}

\begin{rk}
This theorem is false if $\mathfrak{s}$ has type $C$. For example, if it has type $C_2=B_2$, then $\mathfrak{g}=\mathfrak{n}^+$ is 4-dimensional with basis $\{x_0,x_1,x_2,x_3\}$ such that $[x_0,x_2]=x_3$, $[x_0,x_3]=x_1$, all other brackets zero. In this case $B(\mathfrak{g})=\text{ Span}\{x_1,x_2,x_3\}$ and $Z(\mathfrak{g})=\text{ Span}\{x_1\}$.
\end{rk}

\begin{cor}\label{cor: new cases}
Suppose $G$ is a $p$-valuable group and $\mathfrak{g}_G$ is a positive nilpotent subalgebra of a split semisimple algebraic Lie algebra over $\mathbb{Q}_p$, such that no irreducible component has type $C$. Then all faithful prime ideals in $KG$ are controlled by $Z(G)$.
\end{cor}

\begin{proof}
Using Theorem \ref{thm: B(g) for g nilpotent subalgebra}, we know that $Z(\mathfrak{g}_G)=B(\mathfrak{g}_G)$. So using Corollary \ref{cor: when B(G)=Z(G)} we see that all faithful prime ideals in $KG$ are controlled by $Z(G)$.\end{proof}

\subsection{Comparison to $A(G)$}

To assess how much of an improvement the method of using $B(G)$ is to studying prime ideals in $KG$, we need to compare it to results obtained using $A(G)$.

As stated previously, if $A(G)=B(G)$ then this method constitutes no improvement, and this is the case for many examples, as the following result demonstrates.

\begin{lem}\label{lem: type C}
Suppose $\mathfrak{s}$ is a simple Lie algebra of type $C_n$ for $n\geq 2$, and $\mathfrak{g}=\mathfrak{n}^+$. Then $B(\mathfrak{g})$ is a maximal abelian ideal in $\mathfrak{g}$ of dimension $\binom{n+1}{2}$.
\end{lem}

\begin{proof}

We can take the root system of type $C_n$ to be the the set $\Phi$ of all integer vectors in $\mathbb{R}^n$ of length $\sqrt{2}$, together with all vectors of the form $2\underline{v}$, where $\underline{v}$ is an integer vector of length 1. Take $$\Phi^+=\{e_i-e_j,2e_k,e_\ell+e_m:i,j,n,k,\ell,m\leq n, i<j,\ell\neq m\}$$

So clearly the set $A=\{e_i+e_j:i,j\leq n\}=\{2e_k,e_\ell+e_m:k,\ell,m\leq n,\ell\neq m\}\subseteq\Phi^+$, and $A$ has size $\binom{n}{2}+n=\frac{n(n-1)}{2}+n=\frac{n(n+1)}{2}=\binom{n+1}{2}$.

Moreover, for all $\alpha,\beta\in A$, $\alpha+\beta\notin\Phi$, so $[\mathfrak{s}_\alpha,\mathfrak{s}_\beta]=0$. In fact, if $\alpha\in\Phi^+$ and $\alpha\notin A$ then $\alpha=e_i-e_j$ for some $i<j$, so for any $\beta=e_s+e_r\in A$, $\alpha+\beta$ lies in $\Phi^+$ if and only if $j=s$ or $r$. Hence $\alpha+\beta=e_i+e_s$ or $e_i+e_r$, and thus $\alpha+\beta\in A$. So either $[\mathfrak{s}_\alpha,\mathfrak{s}_\beta]=0$ or $[\mathfrak{s}_\alpha,\mathfrak{s}_\beta]=s_\gamma$ for some $\gamma\in A$. It follows that $\mathfrak{h}=\underset{\alpha\in A}{\bigoplus}\mathfrak{s}_\alpha$ is an abelian ideal of $\mathfrak{g}$.

Now, suppose $\mathfrak{a}$ is an abelian ideal containing $\mathfrak{h}$. Then for any element $x\in\mathfrak{a}$, write $x=\underset{\alpha\in\Phi^+}{\sum}\lambda_\alpha x_\alpha$, where $x_{\alpha}$ is a generator of $\mathfrak{s}_{\alpha}$. Therefore, $$[x,x_{2e_j}]=\underset{i<j\leq n}{\sum}\lambda_{e_i-e_j}\mu_{e_i-e_j,2e_j}x_{e_i+e_j}$$ and since $x_{2e_j}\in\mathfrak{h}\subseteq\mathfrak{a}$ and $\mathfrak{a}$ is abelian, we know that $[x,x_{2e_j}]=0$, and hence $\lambda_{e_i-e_j}=0$ for all $i<j\leq n$. It follows that $x\in\mathfrak{h}$ and thus $\mathfrak{a}=\mathfrak{h}$, so $\mathfrak{h}$ is a maximal abelian ideal of $\mathfrak{g}$. 

So to prove that $\mathfrak{h}=B(\mathfrak{g})$, it remains to show that $\mathfrak{h}$ satisfies \eqref{eqn: 2star}. We will show that for all $x\in\mathfrak{g}$, if $[x,[x,\mathfrak{s}_{2e_k}]]=0$ for all $k\leq n$ then $x\in\mathfrak{h}$, and it will follow immediately that if $[x,[x,\mathfrak{h}]]=0$ then $[x,\mathfrak{h}]=0$ as required.

We may assume that the $x_{\alpha}$-coefficient of $x$ is zero if $\alpha\in A$, so we can write $x=\underset{i<j\leq n}{\sum}\lambda_{i,j}x_{e_i-e_j}$, and we want to prove that $\lambda_{i,j}=0$ for all $i,j$.

For each $k\leq n$, we see $$[x,x_{2e_k}]=\underset{i<j\leq n}{\sum}\lambda_{i,j}[x_{e_i-e_j},x_{2e_k}]=\underset{i< k}{\sum}\lambda_{i,k}\mu_{e_i-e_k,2e_k}x_{e_i+e_k}$$ thus 
\begin{align*}
0=[x,[x,x_{2e_k}]]&=\underset{i< k}{\sum}\lambda_{i,k}\mu_{e_i-e_k,2e_k}[x,x_{e_i+e_k}]=\underset{\ell<m\leq n}{\sum}\underset{i< k}{\sum}\lambda_{\ell,m}\lambda_{i,k}\mu_{e_i-e_k,2e_k}[x_{e_\ell-e_m},x_{e_i+e_k}]\\&=\underset{\ell<i<k}{\sum}\lambda_{\ell,i}\lambda_{i,k}\mu_{e_i-e_k,2e_k}\mu_{e_\ell-e_i,e_i+e_k}x_{e_\ell+e_k}+\underset{\ell,i<k}{\sum}\lambda_{\ell,k}\lambda_{i,k}\mu_{e_i-e_k,2e_k}\mu_{e_\ell-e_k,e_i+e_k}x_{e_\ell+e_i}
\end{align*}

And note that the $\mu$'s in this equation are all non-zero. Suppose $k$ is minimal such that $\lambda_{j,k}\neq 0$ for some $j<k$. Then the first sum in the equation above is zero (since every term is a multiple of $\lambda_{\ell,i}$ for $i<k$), and the $x_{2e_{j}}$ coefficient of the second sum is $\lambda_{j,k}^2\mu_{e_\ell-e_k,2e_k}\mu_{e_{\ell}-e_k,e_\ell+e_k}$. So since this coefficient must be zero we conclude that $\lambda_{j,k}=0$ -- contradiction. Thus we must have that $\lambda_{i,j}=0$ for all $i,j$ as required.\end{proof}

\begin{cor}\label{cor: type C}
If $G$ is a $p$-valuable group and $\mathfrak{g}=\mathfrak{g}_G$ is a positive nilpotent subalgebra of a split simple algebraic Lie algebra of type $C_n$, then $B(G)=A(G)$ has rank $\binom{n+1}{2}$, while $Z(G)$ has rank 1.
\end{cor}

\begin{proof}

We know $\mathfrak{g}_{A(G)}$ is an abelian ideal of $\mathfrak{g}$ and $B(\mathfrak{g})=\mathfrak{g}_{B(\mathfrak{g})}\subseteq\mathfrak{g}_{A(G)}$. But $\mathfrak{g}_{B(\mathfrak{g})}$ is a maximal abelian ideal of dimension $\binom{n+1}{2}$ by Lemma \ref{lem: type C}, so it follows that $\mathfrak{g}_{B(\mathfrak{g})}=\mathfrak{g}_{A(\mathfrak{g})}$, and hence $A(G)=B(G)$, and they both have rank $\binom{n+1}{2}$. 

But $Z(\mathfrak{g})=\mathfrak{g}_{Z(G)}=s_{\theta}$ has dimension 1, so $Z(G)$ has rank 1.\end{proof}

\noindent So the $B(G)$ method affords no improvement if we assume that $G$ has type $C$, but in light of Corollary \ref{cor: new cases}, we know that it produces solid results in all other types. The only question remaining is, could a result such as this could be proved using the original $A(G)$-method?\\

The short answer to this is no, because there are examples of groups $G$ given by Corollary \ref{cor: new cases} such $A(G)$ is strictly larger than the centre, so the original control theorem (Theorem \ref{thm: faithful primes controlled by A(G)}) gives us a strictly weaker result. 

For example, if $G$ is the group of unipotent, strictly upper triangular $n\times n$ matrices over $\mathbb{Z}_p$, then $\mathfrak{g}_G$ is a positive nilpotent subalgebra of type $A_{n-1}$, but $A(G)$ consists of all matrices of the form $\left(\begin{array}{cc}
I_k & M \\
  0   & I_{n-k}
\end{array}\right)$ where $k=\lfloor\frac{n}{2}\rfloor$. If $n>3$ then this is never central, so Theorem \ref{thm: faithful primes controlled by A(G)} does not imply our Conjecture A for this group, and we need to apply Corollary \ref{cor: new cases}.

\noindent The following result generalises this example.

\begin{propn}
If $\mathfrak{g}_G$ is the positive subalgebra of a split simple algebraic Lie algebra, then $A(G)\neq Z(G)$.
\end{propn}

\begin{proof}

Let $\Phi$ be the root system of the simple Lie algebra $\mathfrak{s}$, $\Phi^+$ the corresponding positive system, so $\mathfrak{g}_G=\underset{\alpha\in\Phi^+}{\bigoplus}{s_{\alpha}}$.

Using Lemma \ref{Z_2 in A}, we only need to check that $Z_2(G)$ is abelian, or equivalently $Z_2(\mathfrak{g}_G)$ is abelian. But if $\theta\in\Phi^+$ is the highest root then $Z(\mathfrak{g}_G)=s_{\theta}$, so $Z_2(\mathfrak{g}_G)=\underset{\alpha\in \mathcal{Z}}{\bigoplus}\mathfrak{s}_\alpha$, where $$\mathcal{Z}:=\{\beta\in\Phi^+:\text{ if }\alpha\in\Phi^+\text{ and }\alpha+\beta\in\Phi\text{ then }\alpha+\beta=\theta\}$$ 

So $Z_2(\mathfrak{g}_G)$ is abelian if and only if for all $\beta,\gamma\in\mathcal{Z}$, $\beta+\gamma\notin\Phi$. We will calculate $\mathcal{Z}$ for each root system in turn, using the same positive systems that we used in the proof of Lemma \ref{lem: eliminating root systems} (we leave out type $C$, since we calculated $A(G)$ in this case in Corollary \ref{cor: type C}).

\begin{enumerate}
    \item Type $A_n$: $\theta=e_1-e_n$, and $\mathcal{Z}=\{e_1-e_n,e_1-e_{n-1},e_2-e_n\}$.

    \item Type $B_n$: $\theta=e_1+e_2$, and $\mathcal{Z}=\{e_1+e_2,e_1+e_3\}$.

    \item Type $D_n$: $\theta=e_1+e_2$, and $\mathcal{Z}=\{e_1+e_2,e_1+e_3\}$.

    \item Type $E_8$: $\theta=e_1-e_8$ and $\mathcal{Z}=\{e_1-e_8,e_2-e_8\}$.

    \item Type $F_4$: $\theta=e_1-e_4$ and $\mathcal{Z}=\{e_1-e_4,e_2-e_4\}$.

    \item Type $G_2$: $\theta=3\delta+2\varepsilon$ and $\mathcal{Z}=\{3\delta+2\varepsilon,3\delta+\varepsilon\}$.
\end{enumerate}

\noindent We can also realise the roots of type $E_7$ (resp. type $E_6$) as being the set of all roots in $E_8$ where the first two (resp. 3) coordinates are equal. Therefore, we calculate:

\begin{enumerate}
    \item Type $E_7$: $\theta=e_3-e_8$, and $\mathcal{Z}=\{e_3-e_8,e_4-e_8\}$

    \item Type $E_6$: $\theta=e_4-e_8$ and $\mathcal{Z}=\{e_4-e_8,e_5-e_8\}$
\end{enumerate}

\noindent We can immediately see, in all these examples, that taking the sum of any two vectors in $\mathcal{Z}$ will always leave $\mathcal{Z}$.\end{proof}

\subsection{Obstacles to a complete proof}

While the methods outlined in this paper allow us to establish our Conjecture A for a wide range of new examples, they are still not sufficient to prove it in full generality, because there are still many examples of nilpotent groups $G$ such that $B(G)\neq Z(G)$, as demonstrated by Corollary \ref{cor: type C}, and also the following result.

\begin{lem}\label{lem: abelian-by-one}
Suppose $G$ contains a closed, isolated, abelian normal subgroup $H$ such that $G/H\cong\mathbb{Z}_p$. If $G$ has nilpotency class at least 3 then $B(G)=A(G)=H$ and $B(G)\neq Z(G)$.
\end{lem}

\begin{proof}

We will show that $H$ satisfies \eqref{eqn: star}, and thus $H\subseteq B(G)\subseteq A(G)$ by definition. So since $H$ and $A(G)$ are both closed and isolated by Lemma \ref{lem: properties of B}, it follows that $A(G)=H$ or $A(G)=G$. But $G$ is nilpotent, so $A(G)\neq G$ by Lemma \ref{lem: properties of A}, thus $A(G)=B(G)=H$.

Suppose $g\in G$ and $(g,(g,H))=1$. If $g\in H$ then $(g,H)=1$ since $H$ is abelian, so assume for contradiction that $g\notin H$. Then since $G/H\cong\mathbb{Z}_p$, we know that $g$ and $H$ generate an open normal subgroup $U$ of $G$. Moreover, since $(g,H)\subseteq H$, and $g$ and $H$ both commute with $(g,H)$, it follows that $(U,(g,H))=1$, i.e. $(g,H)$ is central in $U$.

But $(g,g)=1$ so this means that $(g,U)$ is central in $U$ by Lemma \ref{lem: commutators}, and hence $g\in Z_2(U)$. But $Z_2(U)\subseteq Z_2(G)$ by the proof of \cite[Lemma 4.9]{jones-control-theorem}. So $g\in Z_2(G)$, and since $(H,H)=1$ this means that $(U,H)\subseteq Z(G)$, so $H\subseteq Z_2(U)\subseteq Z_2(G)$, and hence $U\subseteq Z_2(G)$. This means that $Z_2(G)$ is an open normal subgroup of $G$, so since it is isolated this implies that $Z_2(G)=G$, contradicting the assumption that $G$ has nilpotency class at least 3.

So $B(G)=A(G)=H$, and $H\neq Z(G)$ since $G$ is non-abelian and $G/H$ has rank 1.\end{proof}

To complete a proof of Conjecture A for nilpotent groups in general, it would be enough to establish that it holds for groups of the form $\mathbb{Z}_p^d\rtimes\mathbb{Z}_p$.\\

We will conclude with some further examples, not arising from Lie theory and not covered by the previous cases. The following result suggests how we can try to employ an inductive strategy in the cases where $A(G)$ or $B(G)$ are too large, in particular circumventing the issue that $A(G)\neq Z(G)$ whenever $Z_2(G)$ is abelian.

\begin{propn}\label{propn: new examples}
Let $\mathfrak{h}$ be an ideal of $\mathfrak{g}:=\mathfrak{g}_G$ containing $\mathfrak{g}_{B(G)}$ (or $\mathfrak{g}_{A(G)}$). Suppose that $B(\mathfrak{h})=Z(\mathfrak{h})=Z_2(\mathfrak{g})$. Then all faithful prime ideals of $KG$ are controlled by $Z(G)$.
\end{propn}

\begin{proof}

Let $U$ be any open, normal subgroup of $G$, then $\mathfrak{g}_U=\mathfrak{g}_G=\mathfrak{g}$, while $\mathfrak{h}=\mathfrak{g}_H$ for some closed, isolated normal subgroup $H$ of $U$. Moreover, $B(\mathfrak{g})=\mathfrak{g}_{B(U)}=\mathfrak{g}_{B(G)}$ by Proposition \ref{propn: B(G) to B(g)}, so $B(U)\subseteq H$, and $B(H)=Z(H)=Z_2(U)$.

We know by Corollary \ref{thm: faithful primes are controlled by B} that all faithful primes in $KU$ are controlled by $B(U)$, and hence by $H$. Let $P$ be a faithful prime ideal of $KU$, and write $P\cap KH = Q_1 \cap \dots \cap Q_n$, where $\{Q_1, \dots, Q_n\}$ is a $G$-orbit of prime ideals of $KH$, and let $Q = Q_i$ be arbitrary. Then $Q$ is also faithful, and since $B(H) = Z(H)$ it follows from Corollary \ref{thm: faithful primes are controlled by B} again that $Q$ is controlled by $Z(H)$. It follows by a faithful flatness argument (cf. \cite[Lemma 5.1]{ardakovInv}) that
\begin{align*}
(P\cap KZ(H))KU &= ((P\cap KH)\cap KZ(H))KH\cdot KU\\
&= (Q_1 \cap \dots \cap Q_n \cap KZ(H))KH\cdot KU\\
&= ((Q_1 \cap KZ(H)) \cap \dots \cap (Q_n \cap KZ(H)))KH\cdot KU\\
&= ((Q_1 \cap KZ(H))KH \cap \dots \cap (Q_n \cap KZ(H))KH)KU\\
&= (Q_1 \cap \dots \cap Q_n)KU\\
&= (P\cap KH)KU = P.
\end{align*}

Hence $P$ is controlled by $Z(H)=Z_2(U)$. Since this is true for all open normal subgroups $U$ of $G$, it follows from Proposition \ref{propn: from Z2 to Z} that every faithful prime ideal of $KG$ controlled by $Z(G)$.\end{proof}

The following example demonstrates how to make use of this result to improve the $A(G)$-method (though in this case the $B(G)$-method also works).

\begin{ex}
Suppose $\mathfrak{g}:=\mathfrak{g}_G$ is the $\mathbb{Q}_p$-Lie algebra generated by $x_1, \dots, x_6$, subject to the relations 
$[x_1,x_2] = x_3$, $[x_1,x_5] = x_6$, $[x_2,x_3] = x_4$, $[x_2,x_4] = x_5$, $[x_3,x_4] = x_6$, and otherwise $[x_i, x_j] = 0$.

Then $\mathfrak{g}$ has nilpotency class 5, and its upper central series is
\begin{align*}
Z(\mathfrak{g}) &= \langle x_6\rangle,\\
Z_2(\mathfrak{g}) &= \langle x_5,x_6\rangle,\\
Z_3(\mathfrak{g}) &= \langle x_4,x_5,x_6\rangle,\\
Z_4(\mathfrak{g}) &= \langle x_3,x_4,x_5,x_6\rangle,\\
Z_5(\mathfrak{g}) &= \langle x_1,x_2,x_3,x_4,x_5,x_6\rangle.
\end{align*}

\begin{enumerate}[label=(\roman*)]
\item Since $Z_2(\mathfrak{g})$ is abelian, it follows from Lemma \ref{Z_2 in A} that $A(G)\neq Z(G)$. We can calculate $A(G)$ explicitly by defining the sequence $A_0\supseteq A_1\supseteq A_2\supseteq$ as a sequence of Lie algebra, as opposed to groups, and the intersection will yield $A(\mathfrak{g})=\mathfrak{g}_{A(G)}$:
\begin{align*}
A_0 &= \langle x_1,x_2,x_3,x_4,x_5,x_6\rangle &
Z_2(A_0) &= \langle x_5,x_6\rangle \\
A_1 &= \langle x_2,x_3,x_4,x_5,x_6\rangle &
Z_2(A_1) &= \langle x_4,x_5,x_6\rangle \\
A_2 &= \langle x_4,x_5,x_6\rangle &
Z_2(A_2) &= \langle x_4,x_5,x_6\rangle \\
A_3 &= \langle x_4,x_5,x_6\rangle && \dots
\end{align*}
and so the sequence has stabilised at $A(\mathfrak{g}) = Z_3(\mathfrak{g})$. Hence all faithful prime ideals $P\lhd KG$ are controlled by $Z_3(G)$ by \cite[Theorem B]{jones-primitive-ideals}. Let $\mathfrak{h} = \langle x_3, x_4, x_5, x_6\rangle$, which has nilpotency class 2, with upper central series  
\begin{align*}
Z(\mathfrak{h}) &= \langle x_5, x_6\rangle,\\
Z_2(\mathfrak{h}) &= \langle x_3,x_4,x_5,x_6\rangle=\mathfrak{h}.
\end{align*}

So since $\mathfrak{h}$ has nilpotency class 2, $B(\mathfrak{h})=Z(\mathfrak{h})=\langle x_5, x_6\rangle=Z_2(\mathfrak{g})$, so we can apply Proposition \ref{propn: new examples} to get that all faithful prime ideals in $KG$ are controlled by $Z(G)$.

\item We can calculate $B(\mathfrak{g})$ as follows. Let $y = \alpha_1 x_1 + \dots + \alpha_6 x_6$ be an arbitrary element of $B(\mathfrak{g})$. Then we calculate:
\begin{itemize}
\item $[x_1,[x_1, y]] = 0$ for any such $y$, and hence $[x_1, y] = \alpha_2 x_3 + \alpha_5 x_6 = 0$, so $\alpha_2 = \alpha_5 = 0$.
\item $[x_4,[x_4, y]] = 0$, and hence $[x_4, y] = -\alpha_3 x_6 = 0$, so $\alpha_3 = 0$.
\item $[x_5,[x_5, y]] = 0$, and hence $[x_5, y] = -\alpha_1 x_6 = 0$, so $\alpha_1 = 0$.
\item $[x_2,[x_2, y]] = 0$, and hence $[x_2, y] = \alpha_4 x_5 = 0$, so $\alpha_4 = 0$.
\end{itemize}
Hence $y = \alpha_6 x_6 \in Z(\mathfrak{g})$, and so $B(\mathfrak{g}) = Z(\mathfrak{g})$. Hence all faithful prime ideals $P\lhd KG$ are controlled by $Z(G)$ by Theorem \ref{thm: faithful primes are controlled by B}.
\end{enumerate}
\end{ex}

\noindent Unfortunately, we do not know any examples of groups $G$ for which the $A(G)$-method succeeds using these refinements, but the $B(G)$-method fails. In fact we conjecture that there are none, i.e. for any group $G$ where the condition of Proposition \ref{propn: new examples} is satisfied, we have $B(G)=Z(G)$.

\bibliography{biblio}
\bibliographystyle{plain}

\end{document}